\documentclass[12pt,a4paper]{article}
\usepackage{amssymb}

\usepackage{amssymb}
\usepackage{hyperref}

\usepackage{graphicx,amssymb,amsfonts,epsfig,a4,amsmath,amsthm,xypic}

\usepackage[T1]{fontenc}
\usepackage[utf8]{inputenc}
\usepackage{xcolor}
\usepackage{tikz}
\usetikzlibrary{arrows.meta}
\usetikzlibrary{calc}

\usepackage{tikz-cd}

\newtheorem{Thm}{Theorem}[section]
\newtheorem{Prop}[Thm]{Proposition}
\newtheorem{Lem}[Thm]{Lemma}
\newtheorem{Cor}[Thm]{Corollary}
\newtheorem{Def}[Thm]{Definition}

\newtheorem{Rem}[Thm]{Remark}
\newtheorem{Ex}[Thm]{Example}
\newtheorem{Conj}[Thm]{Conjecture}

\newtheorem{Exo}[Thm]{Exercise}

\newcommand{\Z}{\mathbb{Z}}

\newcommand{\N}{\mathbb{N}}
\newcommand{\R}{\mathbb{R}}
\newcommand{\C}{\mathbb{C}}
\newcommand{\Q}{\mathbb{Q}}
\newcommand{\K}{\mathbf{K}}

\newcommand{\T}{\mathbb{T}}

\newcommand{\HH}{\mathcal{H}}
\newcommand{\bpr}{\noindent \textbf{Proof}: }

\newcommand{\epr}{~$\blacksquare$}

\newcommand{\ra}{\rightarrow}

\title{The Baum-Connes conjecture: a concise course}
\author{Alain VALETTE\footnote{The author acknowledges support of the Institut Henri Poincar\'e (UAR 839 CNRS-Sorbonne Universit\'e), and LabEx CARMIN (ANR-10-LABX-59-01).}}
\date{\today}

\begin{document}

\maketitle

\begin{abstract} These are the slightly edited notes from a 10-hour course I taught at Institut Henri Poincar\'e in January-February 2025 in the framework of the thematic semester {\it Representation theory and noncommutative geometry}.
\end{abstract}

\tableofcontents

\section{Introduction/Preamble}

The Baum-Connes conjecture (i.e. conjecture \ref{BCconj} below) can be seen as a group-equivariant version of the Atiyah-Singer index theorem. Its importance therefore stems from the fact that, like Atiyah-Singer, it bridges two domains of mathematics: topology and analysis. Moreover, since its inception in 1982\footnote{Although published in 2000, the paper \cite{BC82} was circulated as soon as 1982 as an IHES preprint.},  the Baum-Connes conjecture generated a lot of interesting mathematics!

On a very personal note, Baum-Connes has been fascinating me for about 45 years, and aspects of Baum-Connes have been responsible for most of my mathematical trajectory: from Baum-Connes to K-amenability, from K-amenability to property (T) and to the Haagerup property, from the Haagerup property to $CAT(0)$ cube complexes, from property (T) to expander graphs, from expander graphs to embeddings into Hilbert spaces and more general Banach spaces.

This is to say that I accepted with immense pleasure, when Pierre Clare approached me in the fall of 2024 with the idea of teaching an introductory 10-hour course on Baum-Connes at the beginning of the IHP semester {\it Representation theory and noncommutative geometry}. I thank the organizers of the program, Alexandre Afgoustidis, Anne-Marie Aubert, Pierre Clare, Jan Frahm, Angela Pasquale, and Haluk Seng\"un, for giving me this opportunity. Special thanks are due to Pierre Clare who accompanied me all the way through the production of this set of notes. Thanks are also due to the staff at Institut Henri Poincar\'e, who made these 2 weeks in Paris most enjoyable.

The course was a pleasure to teach, also due to the positive attitude of the audience, who came up with numerous excellent questions. Hopefully this set of notes retains some of the informal atmosphere of the lectures. Since the course was videotaped, the reader can also refer to the videos available on \verb"https://www.carmin.tv/en/collections/2025-t1-representation-theory"

\noindent
\verb"-and-noncommutative-geometry"

One thing I was not certain of before the course, was how much time is needed to explain the actors in Conjecture \ref{BCconj}, namely both sides of the conjecture plus the assembly map. It took me 6 hours, corresponding to sections 2 to 7 of these notes. The next 3 hours were devoted to the Dirac-dual Dirac method and Nishikawa's recent approach (sections 8 and 9). Finally the last hour (corresponding to section 10) was devoted to a few examples of discrete groups for which the assembly map can be explicitly described. Actually the reader will notice, throughout the notes, a definite bias towards discrete groups, due to my taste for geometric group theory. 

I cannot end this Introduction without thanking Maria-Paula Gomez Aparicio and Pierre Julg, with whom I had the pleasure of co-authoring the long survey paper \cite{GoJuVa}, fully devoted to the Baum-Connes conjecture and its mathematical offsprings. It is clear that, without that training, teaching the course at IHP would have been much more difficult. Finally, thanks are due to Claudia Stadlmayr for producing the figures in Tikz, and to an extremely kind referee for a careful reading and a number of constructive comments.


\section{What is the Baum-Connes conjecture?}

References for this section: \cite{BC82, BCH94}.

Let $G$ be a locally compact, second countable group (e.g. a countable discrete group, or a connected Lie group).

\begin{Conj}\label{BCconj} (the {\bf Baum-Connes conjecture}) For $i=0,1$, the assembly map 
\begin{equation}\label{BCstandard}
\mu_G: K_i^G(\underline{E}G)\rightarrow K_i(C^*_r(G))
\end{equation}
is an isomorphism.
\end{Conj}

A good part of these notes will be devoted to defining the objects in (\ref{BCstandard}). The left-hand side $K_i^G(\underline{E}G)$ is called the {\it equivariant K-homology of the classifying space for proper actions} of $G$, and is referred to as the {\it geometric} or {\it topological} side. The right-hand side $K_i(C^*_r(G))$ is called the {\it analytical K-theory of the reduced $C^*$-algebra} of $G$, and is referred to as the {\it analytical} side. The map $\mu_G$ is also referred to as the {\it index map}, or the {\it Baum-Connes map}.

\medskip
Several remarks are in order:

\begin{itemize}
\item $K_i^G(\underline{E}G)$ and $K_i(C^*_r(G))$ are abelian groups, and $\mu_G$ is a group homomorphism.
\item There is one conjecture per group $G$ - the game is to prove it either for a given group, or for a class of groups.
\item As of September 2026, no counter-example is known.
\item Injectivity of $\mu_G$ is known for many groups (e.g. countable linear groups, see \cite{GHW}).
\item Conjecture \ref{BCconj} has been proved for several important classes of groups (e.g. connected Lie groups \cite{CEN}, amenable groups \cite{HiKa}, Gromov hyperbolic groups \cite{Laff}).
\item Surjectivity of $\mu_G$ is NOT known for $G=SL_3(\Z)$.
\item Neither injectivity nor surjectivity is known for the free Burnside group $B(d,n)$ (when it is infinite).
\item Conjecture \ref{BCconj} has interesting applications (two of them will be discussed in section \ref{applications} below).
\end{itemize}

\section{The RHS of Conjecture \ref{BCconj}}

\subsection{Reduced group $C^*$-algebras}

Reference for this subsection: \cite{Dix}

Recall that, for $f,g$ functions on $G$, the convolution $f\ast g$ is defined as 
$$(f\ast g)(s)=\int_G f(t)g(t^{-1}s)\,dt,$$
where $s\in G$ and $dt$ denotes left Haar measure on $G$. The Banach space $L^1(G)$ is a Banach algebra with the convolution product. The {\it left regular representation} $\lambda_G$ of $L^1(G)$ on $L^2(G)$ is the representation by left convolution:
$$\lambda_G(f)\xi=f\ast\xi\:(f\in L^1(G),\xi\in L^2(G)).$$

\begin{Def} The reduced $C^*$-algebra $C^*_r(G)$ is the norm closure of $\lambda_G(L^1(G))$ in the bounded operators on $L^2(G)$.
\end{Def}

\begin{Ex}\label{abelian1} If $G$ is abelian with Pontryagin dual $\hat{G}:=Hom(G,S^1)$, then by Fourier transform $C^*_r(G)$ identifies with $C_0(\hat{G})$, the continuous function vanishing at infinity on $\hat{G}$.
\end{Ex}


\subsection{K-theory for $C^*$-algebras}

References for this subsection: \cite{Black,Kar}

Our aim in this section is, for a $C^*$-algebra $A$, to define abelian groups $K_i(A)\;(i=0,1)$, covariant in $A$. We start with $K_1$ as the group structure is more immediately apparent.

\subsubsection{$K_1(A)$}

Assume first that the $C^*$-algebra $A$ is unital. Let $M_n(A)$ denote the $C^*$-algebra of $n$-by-$n$ matrices with coefficients in $A$, and let $U_n(A)$ denote its unitary group, endowed with the norm topology. Denote by $U_n(A)^0$ the connected component of identity of $U_n(A)$ - so that $U_n(A)^0$ is an open normal subgroup. The embedding $U_n(A)\rightarrow U_{n+1}(A):T\mapsto\left(\begin{array}{cc}T & 0 \\0 & 1\end{array}\right)$ induces a homomorphism $U_n(A)/U_n(A)^0\rightarrow U_{n+1}(A)/U_{n+1}(A)^0$ and we define:

\begin{Def} For $A$ a unital $C^*$-algebra: $$K_1(A):=\lim_{\longrightarrow} U_n(A)/U_n(A)^0.$$
\end{Def}

By polar decomposition, $U_n(A)$ is a deformation retract of the invertible group $GL_n(A)$, so an equivalent definition is: $$K_1(A)=\lim_{\longrightarrow}GL_n(A)/GL_n(A)^0.$$

Observe that $S\in GL_m(A),\,T\in GL_n(A)$ define the same element in $K_1(A)$ if and only if $S,T$ are {\it stably homotopic}, i.e. there exists $N\geq\max\{m,n\}$ large enough such that $S,T$ are homotopic in $GL_N(A)$.

\begin{Ex}\begin{enumerate}
\item For $A=\C$, the unitary group $U_n(\C)$ is connected, hence $K_1(\C)=0$.
\item For $A=C(S^1)$, the continuous functions over the circle, the winding number $w:GL_1(C(S^1))\rightarrow\Z$ induces an isomorphism $GL_1(C(S^1))/GL_1(C(S^1))^0\simeq\Z$. Composing with the determinant, we have that $w\circ\det$ induces an isomorphism $GL_n(C(S^1))/GL_n(C(S^1))^0\simeq\Z$. Finally $K_1(C(S^1))=\Z$.
\end{enumerate}\end{Ex}

\begin{Def} For $A$ a non-unital $C`*$-algebra: $K_1(A):=K_1(A^+)$ where $A^+$ denotes the unitization of $A$.
\end{Def}

\begin{Prop}\label{K1abelian} The group $K_1(A)$ is abelian.
\end{Prop}

\bpr We may assume $A$ unital. For $S\in U_n(A)$ form $\left(\begin{array}{cc}S & 0 \\0 &1_n\end{array}\right)\in U_{2n}(A)$. Then $\left(\begin{array}{cc}S & 0 \\0 & 1_n\end{array}\right)$ is homotopic to $\left(\begin{array}{cc}1_n & 0 \\0 & S\end{array}\right)$ in $U_{2n}(A)$, as 
$$\left(\begin{array}{cc}\cos(\pi t/2) & -\sin(\pi t/2) \\\sin(\pi t/2) & \cos(\pi t/2)\end{array}\right)\left(\begin{array}{cc}S & 0 \\0 & 1_n\end{array}\right)\left(\begin{array}{cc}\cos(\pi t/2) & \sin(\pi t/2) \\-\sin(\pi t/2) & \cos(\pi t/2)\end{array}\right)\; (t\in[0,1])$$
is a homotopy between them.

In view of this, for $S,T\in U_n(A)$, we have that $\left(\begin{array}{cc}ST & 0 \\0 & 1_n\end{array}\right)=\left(\begin{array}{cc}S & 0 \\0 & 1_n\end{array}\right)\left(\begin{array}{cc}T & 0 \\0 & 1_n\end{array}\right)\in U_{2n}(A)$ is homotopic to 
$\left(\begin{array}{cc}1_n & 0 \\0 & S\end{array}\right)\left(\begin{array}{cc}T & 0 \\0 & 1_n\end{array}\right)$, a product of two commuting matrices. This shows that $K_1(A)$ is abelian.

\epr

\subsubsection{$K_0(A)$}

A projection $p$ in a $C^*$-algebra is an idempotent, self-adjoint element ($p=p^2=p^*$). Denote by $Proj(A)$ the set of projections in the $C^*$-algebra $A$. We will say that two projections $p,q$ in $M_n(A)$ are {\it equivalent} ($p\sim q$) if there exists $x\in M_n(A)$ such that $p=x^*x$ and $q=xx^*$.

\begin{Exo} Verify that, if the projections $p,q\in M_n(A)$ are unitarily conjugate, then $p\sim q$.
\end{Exo}

The map $M_n(A)\rightarrow M_{n+1}(A):T\mapsto \left(\begin{array}{cc}T & 0 \\0 & 0\end{array}\right)$ is an embedding (non-unital in case $A$ has a unit!). We set
$$M_\infty(A):=\lim_{\longrightarrow} M_n(A).$$
The set $Proj(M_\infty(A))$ has an addition given by direct sum: for $p\in M_m(A),q\in M_n(A)$, then $p\oplus q :=\left(\begin{array}{cc}p & 0 \\0 & q\end{array}\right)\in M_{m+n}(A)$. Since $p\oplus q$ is unitarily conjugate to $q\oplus p$, we see that $Proj(M_\infty(A))/\sim$ becomes a commutative monoid.

\begin{Def} For $A$ a unital $C^*$-algebra, $K_0(A)$ is the Grothendieck group of the commutative monoid $Proj(M_\infty(A))/\sim$.
\end{Def}

\begin{Ex}\label{ExK0}\begin{enumerate}
\item For $A=\C$, a projection in $M_\infty(\C)$ is characterized up to equivalence by its rank, so $Proj(M_\infty(\C))/\sim$ is isomorphic to $\N,+$. Kence $K_0(\C)=\Z$.
\item If $X$ is a compact space\footnote{For us, being compact implies being Hausdorff.}, and $A=C(X)$, the $C^*$-algebra of continuous functions on $S$, then by the Serre-Swan theorem $K_0(C(X))$ identifies with $K^0(X)$, the Grothendieck group of the commutative monoid of isomorphism classes of complex vector bundles over $X$. In particular, if $X$ is contractible every vector bundle over $X$ is trivial and $K_0(C(X))=\Z$.
\item Consider $X=S^1$: of course $S^1$ is not contractible, however every {\it complex} vector bundle over $S^1$ is trivial\footnote{This would not be true for {\it real} vector bundles over $S^1$, as shown by the open M\"obius band.}.  Indeed, let $E$ be a complex vector bundle over $S^1$, say of rank $n\geq 1$. View $S^1$ as the quotient of $[0,1]$ by identiying endpoints, let $p:[0,1]\rightarrow S^1$ denote the quotient map. As $[0,1]$ is contractible, the pull-back vector bundle $p^*E$ is trivial. So $E$ is described by the identification of the fibre at 0 with the fibre at 1, i.e. by some element $A\in GL_n(\C)$. But by connectedness of $GL_n(\C)$ we may homotope $A$ to the identity, making $E$ isomorphic to the trivial bundle. As a result $K_0(C(S^1))=\Z$.
\end{enumerate}\end{Ex}

\begin{Rem}\label{idem/proj} Assume that, instead of dealing with projections to define $K_0(A)$, we relax the condition of being self-adjoint, and just work with idempotent elements. So let $Idem(M_\infty(A))$ denote the set of idempotents in $M_\infty(A)$; say that idempotents $e,f\in M_n(A)$ are equivalent $(e\sim' f$) if there exist $u,v\in M_n(A)$ such that $e=vu,\; f=uv$. We claim that the monoid $Proj(M_\infty(A))/\sim$ and $Idem(M_\infty(A))/\sim'$ are actually the same. The reason is that any idempotent in a unital $C^*$-algebra is equivalent to a projection. Here is a short proof (taken from \cite{FaMa}): if $e\in A$ is idempotent, define a positive element:
$$z=1+(e^*-e)^*(e^"-e) = 1-(e^*-e)^2=1-e-e^*+ee^*+e^*e.$$
Then $z$ is invertible and $ze=ez=ee^*e$. So $z$ and $z^{-1}$ commute with $e$ and $e^*$. Set then
$$p:=ee^*z^{-1}.$$
Then $p\in Proj(A)$ (check that $p=p^2$; actually $p$ is the support projection of $e^*$). Finally $ep=p$ trivially while $pe= z^{-1}ee^*e=z^{-1}ze=e$, so $e\sim' p$. So using idempotents rather than projections defines the same group $K_0(A)$. This remark will be useful in Proposition \ref{traceKth}.
\end{Rem}

If $A$ is a non-unital $C^*$-algebra, let $\epsilon:A^+\rightarrow\C$ denote the canonical morphism from the unitization to the complex numbers.

\begin{Def} For $A$ a non-unital $C^*$-algebra, define $$K_0(A):=\ker[\epsilon_*:K_0(A^+)\rightarrow K_o(\C)=\Z].$$
\end{Def}

\subsubsection{Properties of the K-groups}
\begin{itemize}
\item (Functoriality) For $i=0,1$, the assignment $A\rightarrow K_i(A)$ is a covariant, homotopy invariant functor from $C^*$-algebras to abelian groups.
\item (Stability) Let $\mathcal{K}$ denote the algebra of compact operators on $\ell^2$. If $p\in \mathcal{K}$ is a rank 1 projection, then for $i=0,1$ the homomorphism 
$$A\rightarrow A\otimes\mathcal{K}:a\mapsto a\otimes p$$
induces an isomorphism on $K_i$.
\item (Bott periodicity) For $i=0,1$, the group $K_i(A)$ is naturally isomorphic to $K_{1-i}(A\otimes C_0(\R))$.
\item (6-terms exact sequence) Let $I$ be a closed, 2-sided ideal in $A$. Then there is a natural 6-terms exact sequence:
$$\begin{array}{ccccc}K_0(I) & \rightarrow & K_0(A) & \rightarrow & K_0(A/I) \\ \uparrow &  &  &  & \downarrow  \\K_1(A/I) & \leftarrow & K_1(A) & \leftarrow & K_1(I)\end{array}$$
where the horizontal arrows are induced by the inclusion $I\hookrightarrow A$ and the quotient map $A \twoheadrightarrow A/I$.
\end{itemize}

\section{Why do we care about the Baum-Connes conjecture?}\label{applications}

In this section, $G$ will denote a countable, discrete group. 

\subsection{The Kaplansky-Kadison conjecture}

\begin{Conj}\label{KapKad} (Kaplansky-Kadison, around 1954) If $G$ is torsion-free, then the only idempotents in $C^*_r(G)$ are 0 and 1.
\end{Conj}

\begin{Ex} Assume $G$ abelian. Then $\hat{G}$ is a compact abelian group, and $C^*_r(G)=C(\hat{G})$ (see Example \ref{abelian1}). So $C^*_r(G)$ has no non-trivial idempotent if and only if $\hat{G}$ is connected; by Pontryagin duality, this happens if and only if $G$ is torsion-free.
\end{Ex}

\begin{Ex}\label{torsiongivesidem} Assume $G$ has torsion, i.e. there exists $g\in G$ with finite order $N\geq 2$. Then $p_g:=\frac{1}{N}\sum_{k=0}^{N-1}\lambda_G(g^k)$ is a non-trivial projection in $C^*_r(G)$. So the assumption in Conjecture \ref{KapKad} is necessary.
\end{Ex}

\begin{Def} Let $A$ be a $C^*$-algebra. A {\bf trace} on $A$ is a continuous linear functional $\tau$ on $A$ such that $\tau(ab)=\tau(ba)$ for every $a,b\in A$. A trace $\tau$ is {\bf positive} if $\tau(a^*a)\geq 0$ for every $a\in A$. A positive trace $\tau$ is {\bf faithful} if $\tau(a^*a)=0$ implies $a=0$. If $A$ is unital, a trace $\tau$ is {\bf normalized} if $\tau(1_A)=1$.
\end{Def}

\begin{Ex} On $C^*_r(G)$, we have the {\bf canonical trace} $\tau_G(a):=\langle a(\delta_e)|\delta_e\rangle$ for $a\in C^*_r(G)$. It is normalized (clear), positive ($\tau_G(a^*a)=\|a(\delta_e)\|_2^2\geq 0$), and faithful. To check the letter, let $\rho_G$ denote the right regular representation of $G$ on $\ell^2(G)$, and observe that $\rho(G)$ commutes with $C^*_r(G)$. Then, if $a\in C^*_r(G)$ is such that $a(\delta_e)=0$, we have for all $g\in G$:
$$0=\rho(g^{-1})a(\delta_e)=a\rho(g^{-1})(\delta_e)=a(\delta_g);$$
so $a$ vanishes on a basis of $\ell^2(G)$, hence $a=0$.
\end{Ex}

\begin{Exo} Check that $\tau_G$ is indeed a trace on $C^*_r(G)$.
\end{Exo}

\begin{Prop}\label{traceKth} Let $A$ be a $C^*$-algebra with unit $1_A$, and $\tau$ be a trace on $A$.
\begin{enumerate}
\item There exists a homomorphism $\tau_*:K_0(A)\rightarrow \C$ such that $\tau_*[1_A]=\tau(1_A)$.
\item If $\tau$ is positive, then $\tau_*(K_0(A))\subset\R$.
\item If $\tau$ is positive, faithful, normalized, and $\tau_*(K_0(A))=\Z$, then $A$ has no non-trivial idempotent. 
\end{enumerate}\end{Prop}

\bpr\begin{enumerate}
\item Define
$$Tr:M_n(A)\rightarrow A:a\mapsto \sum_{i=1}^n a_{ii}.$$
The map $Tr$ extends to $M_\infty(A)$; observe that $\tau\circ Tr$ is a trace on $M_\infty(A)$. So, for $p\in Proj(M_\infty(A))$ we may define $\tau_*[p]=(\tau\circ Tr)(p)\in\C$. It is easy to check that this map is well-defined and additive on $Proj(M_\infty(A))/ {\sim}$, hence extends to the desired homomorphism $\tau_*$.
\item If $\tau$ is positive, then for $p\in Proj(M_\infty(A)$ we have $\tau_*[p]=(\tau\circ Tr)(p)=(\tau\circ Tr)(p^*p)\geq 0$.
\item Assume that $\tau$ is positive, faithful, normalized, with $\tau_*(K_0(A))=\Z$. Let $e$ be an idempotent in $A$. By Remark \ref{idem/proj}, there exists $p\in Proj(A)$ such that $e\sim' p$; observe that $\tau(e)=\tau(p)$. So it is enough to prove that either $p=0$ or $p=1_A$. Now $\tau(p)=\tau(p^*p)\geq 0$, and similarly for $1-\tau(p)=\tau(1_A-p)$. So $0\leq\tau(p)\leq 1$. Since $\tau(p)$ is an integer by assumption, we have $\tau(p)\in\{0,1\}$. If $0=\tau(p)=\tau(p^*p)$, then $p=0$ by faithfulness. If $\tau(p)=1$, replace $p$ by $1_A-p$ to similarly get $p=1_A$.\epr
\end{enumerate}

\begin{Conj}\label{Integraltrace} (Integrality of the trace) If $G$ is torsion-free, then $$(\tau_G)_*(K_0(C^*_r(G)))=\Z.$$
\end{Conj}

Note that, if $g\in G$ has order $N\geq 2$, then $\tau_G(p_g)=\frac{1}{N}$ by Example \ref{torsiongivesidem}. So the assumption is also necessary in Conjecture \ref{Integraltrace}.
In view of Proposition \ref{traceKth}, Conjecture \ref{Integraltrace} implies Conjecture \ref{KapKad}.

\begin{Prop}\label{BaumDouglas} If $G$ is torsion-free, then $$(\tau_G)_*(\mu_G(K_0^G(\underline{E}G)))=\Z.$$
\end{Prop}

From Proposition \ref{BaumDouglas} we deduce immediately:

\begin{Cor} If $G$ is torsion-free and $\mu_G$ is onto, then Conjecture \ref{Integraltrace} holds. In particular, Conjecture \ref{BCconj} implies Conjecture \ref{Integraltrace}.
$\blacksquare$
\end{Cor}

{\bf Sketch of proof of Proposition \ref{BaumDouglas}:} This proof will allow to get a first feeling of the LHS of Conjecture \ref{BCconj}. Let $BG$ be a classifying space for $G$, i.e. a CW-complex with fundamental group $G$ such that the universal cover $EG:=\widetilde{BG}$ is contractible\footnote{Topologists also say that $BG$ is a $K(G,1)$-Eilenberg-MacLane space, or an aspherical space with fundamental group $G$.}; such a space $BG$ always exists and is unique up to homotopy; see also Example \ref{universalex} below. Since $G$ is torsion-free, we may take $\underline{E}G=EG$, and since $G$ acts freely on $EG$ we have 
$$K_i^G(\underline{E}G)=K_i(BG)$$
for $i=0,1$.  Now, as we will see in Definition \ref{LHS}:
$$K_i(BG)=\lim_{X\subset BG} KK_i(C(X),\C),$$
where the limit is on the directed set of compact subsets $X$ in $BG$, and $KK_i$ denotes Kasparov bivariant K-theory (see subsection \ref{KKordinary}); for $A$ a $C^*$-algebra, $K_i(A,\C)$ is the analytical $K$-homology of $A$, see \cite{HiRo}.

By a result of P. Baum and R.G. Douglas \cite{BaDo}, an element $x\in KK_0(C(X),\C)$ can be described as a quadruple $(M,E,D,f)$ such that 
\begin{itemize}
\item $M$ is a smooth closed manifold (i.e. compact without boundary).
\item $E=E^+\oplus E^-$ is a smooth vector bundle over $M$.
\item $D:C^\infty(M,E^\pm)\rightarrow C^\infty(M,E^\mp)$ is an elliptic differential operator acting on smooth sections of $E^\pm$.
\item $f:M\rightarrow X$ is a continuous map.
\end{itemize}
At this point we observe that the triple $(M,E,D)$ defines an element $y\in KK_0(C(M),\C)$ such that $f_*(y)=x$. Ellipticity of $D$ guarantees that the spaces $Ker(D)$ and $Coker(D)$ are finite-dimensional, so we have the Fredholm index (see Definition \ref{Fred} below):
$$Ind(D):=\dim_\C Ker(D)-\dim_\C Coker(D)$$
that defines the index map $Ind: KK_0(C(M),\C)\rightarrow \Z$.

Now form the pull-back
$$\left.\begin{array}{ccc}\tilde{M} & \rightarrow & EG \\ p\downarrow &  & \downarrow \\M & \stackrel{f}{\rightarrow} & BG\end{array}\right.$$
so that $\tilde{M}$ is a proper $G$-manifold. Let $\tilde{E}:=p^*E$ denote the pulled-back vector bundle, and $\tilde{D}$ the lift of $D$, so that $\tilde{D}$ is a $G$-invariant elliptic differential operator acting on sections of $\tilde{E}$. Since $\tilde{M}$ is non-compact as soon as $G$ is infinite, classical ellipticity theorems fail, and $Ker(\tilde{D}),\,Coker(\tilde{D})$ are usually infinite-dimensional. However $Ker(\tilde{D}),\,Coker(\tilde{D})$  are modules over $C^*_r(G)$ and it is a result by A. Connes and H. Moscovici \cite{CoMo} that the formal difference of those modules provides a well-defined element $Ind_G(\tilde{D})$, called the $G$-index of $\tilde{D}$, in the K-theory of the reduced $C^*$-algebra\footnote{Although this is not true, one may morally think that $Ker(\tilde{D}),\,Coker(\tilde{D})$ are finitely generated projective modules over $C^*_r(G)$.}. Moreover 
$$Ind_G(\tilde{D})=\mu_G(x).$$
By Atiyah's $L^2$-index theorem (see \cite{Ati}):
$$(\tau_G)_*(Ind_G(\tilde{D}))=Ind(D),$$
which amounts to saying that the following diagram is commutative:
$$\left.\begin{array}{ccc}KK_0(C(M),\C) & \stackrel{\mu_G}{\longrightarrow} & K_0(C^*_r(G)) \\ & & \\Ind\downarrow&  & \downarrow(\tau_G)_* \\ & & \\ \Z & \hookrightarrow & \R\end{array}\right.$$
The result is now clear. $\blacksquare$

\subsection{A Whitehead group-type conjecture}

Recall that a group with a finite-dimensional classifying space, is necessarily torsion-free.

\begin{Conj}\label{Whitehead} Assume that $G$ admits a 2-dimensional model for $BG$. Then any invertible $T\in GL_n(A)$ is stably homotopic to a $N$-by-$N$ diagonal matrix
$$\left(\begin{array}{ccccc}g & 0 & 0 & \cdots & 0 \\0 & 1 & 0 & \cdots & 0 \\0 & 0 & 1 & \cdots & 0 \\ \vdots & \vdots & \vdots & \ddots & \vdots \\0 & 0 & 0 & \cdots & 1\end{array}\right)$$
for some $N\geq n$ and some $g\in G$.
\end{Conj}

This conjecture is again a consequence of surjectivity in Conjecture \ref{BCconj}.

\begin{Prop} Assume that $G$ admits a 2-dimensional $BG$, and $\mu_G$ is onto. Then Conjecture \ref{Whitehead} holds.
\end{Prop}

{\bf Sketch of proof:} As $G$ is discrete, the canonical map
$$G\rightarrow C^*_r(G):g\mapsto\lambda_G(g)$$
embeds $G$ into the unitary group $U_1(C^*_r(G))$. Composing with the canonical map $U_1(C^*_r(G))\rightarrow K_1(C^*_r(G))$, and using commutativity of $K_1$ (Proposition \ref{K1abelian}), we get the canonical map $\iota:G/[G,G]\rightarrow K_1(C^*_r(G))$.

Since $G$ is torsion-free, as in the proof of Proposition \ref{BaumDouglas} , the  LHS of Conjecture \ref{BCconj} can be replaced by $K_i(BG)$. As $BG$ is 2-dimensional, we have a Chern character $Ch_*:K_1(BG)\rightarrow H_1(BG,\Z)$ which is an isomorphism, where $H_1$ denotes the usual cellular homology. Moreover the following diagram is commutative (see e.g. \cite{Matthey}):
$$\left.\begin{array}{ccc}K_1(BG) & \stackrel{\mu_G}{\longrightarrow} & K_1(C^*_r(G) \\ &  &  \\Ch_*\downarrow &  & \uparrow\iota \\ &  &  \\H_1(BG,\Z) & = & G/[G,G]\end{array}\right.$$
Since $Ch_*$ is an isomorphism and $\mu_G$ is assumed to be onto, the map $\iota$ is onto, which concludes the proof. $\blacksquare$

\begin{Rem} The classical Whitehead group of a discrete group $G$ is defined as $Wh(G):= K_1^{alg}(\Z G)/\langle\pm g: g\in G\rangle$, where $K_1^{alg}(\Z G)$ denotes the algebraic K-theory of the integral group ring of $G$. Define in a similar way the analytical Whitehead group of $G$ as $Wh^{an}(G):=K_1(C^*_r(G))/Im(\iota)$; then Conjecture \ref{Whitehead} is equivalent to the vanishing of $Wh^{an}(G)$ whenever $G$ has a 2-dimensional $BG$.
\end{Rem}

\section{Proper actions}

Reference for this section: \cite{BCH94}.

Here we come back to our standard assumption on $G$: a locally compact, second countable group. 

\subsection{Definition and examples}

Let $X$ be a metrizable space.  An {\it action} $G\curvearrowright X$ of $G$ on $X$ is a homomorphism from $G$ to the homeomorphism group of $X$ such that the action map $G\times X\rightarrow X:(g,x)\mapsto g\cdot x$ is continuous.

\begin{Def}\label{properaction} An action $G\curvearrowright X$ is {\bf proper} if for every point $x\in X$ there exists a $G$-invariant open neighborhood $U$ of $x$, a compact subgroup $H$ of $G$ and a $G$-equivariant continuous map $f_x:U\rightarrow G/H$.
\end{Def}

So a proper action is a an action locally modeled on a homogeneous space $G/H$ with compact isotropy $H$.

\begin{Ex}\begin{enumerate}
\item Any action of a compact group $G$ on any space $X$, is proper: take $U=X,\,H=G$ and $f_x$ the constant map to the one-point space.
\item Let $D_\infty$ denote the infinite dihedral group, i.e. the group given by the presentation
$$D_\infty=\langle s,t|s^2=t^2=e\rangle$$
(so that $D_\infty$ is the free product of two copies of $\Z/2\Z$). Let $D_\infty$ act on the real line by setting $s(x)=-x$ (symmetry about 0) and $t(x)=2-x$ (symmetry about 1). There are now three cases:
\begin{itemize}
\item For $x\in\R\setminus\Z$, we may take $U=\R\setminus\Z, \,H=\{e\}$; then, denoting by $U_x$ the connected component of $x$ in $U$, every connected component of $U$ is of the form $gU_x$ for a unique $g\in D_\infty$. So we may take as $f_x:U\rightarrow D_\infty$ the map that sends $gU_x$ to $g$.
\item For $x\in 2\Z$, assume first $x=0$. Then take $U=\R\setminus(2\Z+1)$ and $H=\langle s\rangle$. Denote by $U_{2n}$ the connected component of $U$ through $2n$; observe that there are exactly two elements $g\in D_\infty$ such that $gU_0=U_{2n}$, namely $g=(ts)^n$ and $g=(ts)^ns$. This defines the map $f_0:U\rightarrow D_\infty/\langle s\rangle$. For a general $x\in 2\Z$, just use conjugation.
\item For $x\in 2\Z+1$, assume first $x=1$, set $U=\R\setminus 2\Z$ and $H=\langle t\rangle$, and proceed analogously. 
\end{itemize}
\item For $G$ discrete, consider a covering $\begin{array}{l}X \\ \downarrow G \\Y\end{array}$ with group $G$: local triviality allows to write $X$ locally as $V\times G$, where $V$ is an open set in the basis $Y$. Here the action $G\curvearrowright X$ is both proper and free.
\item If $X$ is a CW-complex and $G\curvearrowright X$ is cellular, the action is proper if and only if stabilizers of vertices of $X$ are compact.
\end{enumerate}
\end{Ex}

\subsection{The classifying space for proper actions}

\begin{Def}\cite{BCH94} A {\bf universal example} for $G$-proper actions is a proper $G$-space $\underline{E}G$ such that for any proper $G$-space $X$ there exists a continuous $G$-equivariant map $f:X\rightarrow \underline{E}G$, and any two such maps are $G$-equivariantly homotopic.
\end{Def}

In more categorical language: $\underline{E}G$ is a final object in the category of proper $G$-spaces and $G$-equivariant maps taken up to $G$-equivariant homotopy.

If $\underline{E}G,\,(\underline{E}G)'$ are two universal examples for $G$-proper actions, let $f:(\underline{E}G)'\rightarrow\underline{E}G$ and $f':\underline{E}G\rightarrow(\underline{E}G)'$ be the $G$-equivariant maps given by the universal property. By uniqueness up to homotopy, $f\circ f'$ is homotopic to the identity of $\underline{E}G$, and $f'\circ f$ is homotopic to the identity of $(\underline{E}G)'$. Hence $\underline{E}G$ is unique up to $G$-equivariant homotopy.  By a slight abuse of terminology, we call $\underline{E}G$ {\it ``the'' classifying space for proper actions} of $G$.

An intrinsic characterization of universal examples is given by the next result.

\begin{Prop}\label{universal} A proper $G$-space $Y$ is a universal example for $G$-proper actions if and only if the two following conditions are satisfied:
\begin{enumerate}
\item For every compact subgroup $H$ of $G$, the fixed point set $Y^H$ is non-empty.
\item Equip $Y\times Y$ with the diagonal $G$-action; let $p,p':Y\times Y\rightarrow Y$ denote the two projections. Then $p,\,p'$ are $G$-equivariantly homotopic.
\end{enumerate}\end{Prop}

{\bf Proof of necessity:} If $Y$ is a universal example, then:
\begin{enumerate}
\item If $H$ is a compact subgroup, then $G/H$ is a proper $G$-space, so there exists a $G$-equivariant map $f:G/H\rightarrow Y$. Since the coset $eH$ is $H$-fixed in $G/H$, we have $f(eH)\in Y^H$.
\item The diagonal $G$-action on $Y\times Y$ is proper. Since the maps $p,\,p'$ are $G$-equivariant, by the universal property of $Y$ they must be $G$-equivariantly homotopic. $\blacksquare$
\end{enumerate}

\begin{Cor} If $H$ is any closed subgroup of $G$, then $\underline{E}G$ is a universal example for $H$-proper actions.
\end{Cor}

\begin{Prop} Every locally compact, second countable group $G$ admits a universal example for $G$-proper actions.
\end{Prop}

{\bf Proof when $G$ is discrete, countable:} Define $\underline{E}G$ as the set of finitely supported probability measures on $G$, with the distance 
$$d(\mu,\nu)=\|\mu-\nu\|_\infty=\sup_{g\in G}|\mu(g)-\nu(g)|,$$
and the $G$-action by left translations. It is clear that this is a proper $G$-space. Let us check that it satisfies the two conditions in Proposition \ref{universal}.
\begin{enumerate}
\item If $H$ is a finite subgroup of $G$, then the uniform probability on $H$ is $H$-fixed.
\item For $0\leq t\leq 1$, the map $(\mu,\nu)\mapsto (1-t)\mu+t\nu$ provides a $G$-equivariant homotopy between $p$ (at $t=0$) and $p'$ (at $t=1$). $\blacksquare$
\end{enumerate}

\begin{Ex}\label{universalex}\begin{enumerate}
\item If $G$ is compact, then one may take $\underline{E}G$ to be a point. To anticipate on Example \ref{compactgroup}, in this case $K_0^G(\underline{E}G)=R(G)$, the additive group of the representation ring of $G$ (equivalently, the free abelian group on the dual $\hat{G}$); and $K_1^G(\underline{E}G)=0$.
\item If $G$ is a connected Lie group, let $K$ be a maximal compact subgroup of $G$ (recall that $K$ is unique up to conjugacy), we may take $\underline{E}G=G/K$. In particular, if $G$ is semi-simple with finite center, $G/K$ is the Riemannian symmetric space associated with $G$.
\item Let $F$ be a non-Archimedean local field (e.g. $F=\Q_p$, the $p$-adic numbers), and let $G$ be a semi-simple algebraic group over $F$ (e.g. $G=SL_n(F)$). Then we may take for $\underline{E}G$ the Bruhat-Tits building associated with $G$.
\item Let $G$ be a discrete torsion-free group, so that any proper action is also free. Denoting by $EG$ a contractible free $G$-space\footnote{Quoting from A. Connes: ``$EG$ is a point on which $G$ acts freely!''}, we may take $\underline{E}G=EG$.  Setting $BG:=EG/G$, we see that $EG$ is the universal cover of $BG$, so that $BG$ is the classifying space for $G$. Then $K_i^G(\underline{E}G)=K_i(BG)\,(i=0,1)$.
\item Let $G$ be a countable locally finite group, i.e. $G$ is a countable increasing union of finite subgroups $F_n$, say $G=\cup_{n\geq 0}F_n$. Then $\underline{E}G$ can be taken as a tree with vertex set $\coprod_{n\geq 0}G/F_n$ and edges of the form $\{gF_n,gF_{n+1}\}$, for $g\in G$. Observe that this tree fibers over $\N$.
\end{enumerate}\end{Ex}

\section{The LHS of Conjecture \ref{BCconj}}

\subsection{Motivation 1: K-homology vs. K-theory}

Topological K-theory is a cohomology theory on spaces. For a compact space $X$, the group $K^0(X)$ is the group of formal differences of isomorphism classes of complex vector bundles on $X$, as we saw already in Example \ref{ExK0}:
$$K^0(X)=\{[E]-[F]:E,F \;\mbox{vector bundles over}\; X\}.$$
For abstract reasons, to any cohomology theory there is a dual homology, defined by Spanier-Whitehead duality. The dual of K-theory is called {\it K-homology}. 

More precisely, for $X$ a finite CW-complex, one may embed $X$ into a $2N$-dimensional sphere $S^{2N}$ for $N\gg 0$, and one defines K-homology of $X$ as K-theory of $S^{2N}\setminus U$, where $U$ is a tubular neighborhood of $X$ in $S^{2N}$. However this definition is not very satisfactory, as it is not clear that it is well-defined: indeed it must be checked that this definition is independent of $N$ and independent of the choice of the embedding; moreover functoriality is not apparent. Therefore a more intrinsic definition is desirable.

\subsection{Motivation 2: Atiyah's $Ell(X)$}\label{AtiyahEll}

Atiyah \cite{AtiEll} suggested a direct construction cycles in K-homology.

\subsubsection{Fredholm operators}

Let $\HH$ be a Hilbert space (infinite-dimensional, separable) and $\mathcal{B}(\HH)$ denote the $C^*$-algebra of bounded, linear operators on $\HH$.

\begin{Def}\label{Fred} An operator $T\in\mathcal{B}(\HH)$ is a {\bf Fredholm operator} if the image $Im(T)$ is closed and the spaces $Ker(T)$ and $Coker(T):=\HH/Im(T)$ are finite-dimensional. The {\bf index} of $T$ is then defined as $$Ind(T)=\dim_\C Ker(T)-\dim_\C Coker(T).$$
\end{Def}

Equivalently, a bounded operator $T$ is Fredhom if $T$ is invertible modulo compact operators (Atkinson's theorem). If $T$ is Fredholm, then $Coker(T)$ identifies with $Im(T)^\perp=Ker(T^*)$.

Denote by $Fred(\HH)$ the set of Fredholm operators on $\HH$, endowed with the operator norm topology. Observe that, for $S,T\in Fred(\HH)$, we have $S\oplus T\in Fred(\HH\oplus\HH)$ and $Ind(S\oplus T)=Ind(S)+Ind(T)$. For a compact space $X$, denote by $[X,Fred(\HH)]$ the set of homotopy classes of continuous maps $X\rightarrow Fred(\HH)$. Composing direct sum with any identification $\HH\oplus\HH\simeq\HH$, we equip $[X,Fred(\HH)]$ with the structure of a commutative monoid.

\begin{Thm}\label{AtiyahK0}(Atiyah) For $X$ compact, $K^0(X)\simeq [X,Fred(\HH)]$ as commutative monoids.
\end{Thm}

The proof of Theorem \ref{AtiyahK0} rests on a simple idea: if $T:X\rightarrow Fred(\HH)$ is a continuous map, it is tempting to consider the two fields of finite-dimensional vector spaces $(Ker(T_x))_{x\in X},(Coker(T_x))_{x\in X}$ and take the formal difference as an element of $K^0(X)$. The difficulty is that those two fields of vector spaces are not necessarily vector bundles over $X$ (local triviality may fail), so some work is necessary to prove that the formal difference is indeed well-defined as an element of $K^0(X)$ (see \cite{AtiKth}).

\subsubsection{Elliptic operators and Fredholm modules}

Let $M$ be a closed manifold, and let $P$ be an elliptic differential operator, say of order $m$, acting on sections on some (possibly trivial) vector bundle over $M$. Introduce the Sobolev spaces $H^s$, consisting of sections having square-integrable partial derivatives up to order $m$. Then $P:H^s\rightarrow H^{s-m}$ is bounded. By ellipticity it admits a {\it parametrix}, i.e. an inverse modulo lower order operators, and hence $P$ is Fredholm; 
moreover its index is independent of $s\geq m$.

For elliptic pseudo-differential operators of order 0, we may take $s=0$ and work just with Hilbert spaces of square-integrable sections.

\begin{Ex}\label{circle} Take $M=S^1$ and $P=\frac{1}{i}\frac{d}{dt}$ acting on $L^2([0,2\pi],\frac{dt}{2\pi})$. In the trigonometric basis $e_k(t)=e^{ikt}\;(k\in\Z)$, we have $P^2(e_k)=k^2e_k$ and $(1+P^2)^{1/2}(e_k)=(1+k^2)^{1/2}e_k$. Viewing $(1+P^2)^{1/2}$ as an order 1, invertible operator, we have:
$$\frac{P}{(1+P^2)^{1/2}}(e_k)=\frac{k}{(1+k^2)^{1/2}}e_k$$
and $\frac{P}{(1+P^2)^{1/2}}$ is a pseudo-differential operator of order 0. Then 
$$(1-(\frac{P}{(1+P^2)^{1/2}})^2)(e_k)=\frac{1}{1+k^2}e_k$$
so that $1-(\frac{P}{(1+P^2)^{1/2}})^2$ is a compact operator, and $\frac{P}{(1+P^2)^{1/2}}$ is Fredholm (of course of index 0, as it is self-adjoint).
\end{Ex}

\begin{Ex}\label{deRham} Let $M$ be a closed Riemannian manifold. Let $\Lambda^k T^*M$ denote the $k$-th exterior power of the cotangent bundle $T^*M$ (so that sections of $\Lambda^k T^*M$ are exactly $k$-forms on $M$). Set 
$$\Lambda^+ T^*M:=\oplus_{k\geq 0}\Lambda^{2k}T^*M;\;\Lambda^- T^*M:=\oplus_{k\geq 0}\Lambda^{2k+1}T^*M.$$
 Denoting by $d$ the exterior derivative on forms, define an operator $P:L^2(\Lambda^+ T^*M)\rightarrow L^2(\Lambda^- T^*M)$ by $P=d+d^*$. Then $P^*P=d^*d+dd^*$ is the Laplace-de Rham operator on differential forms, and $F:=P(1+P^*P)^{-1/2}$ is an elliptic pseudo-differential operator of order 0. Moreover
$$Ker(F)=\oplus_{k\geq 0}\{\mbox{harmonic $2k$-forms}\};\;Coker(F)=\oplus_{k\geq 0}\{\mbox{harmonic $(2k+1)$-forms}\}.$$
So $Ind(F)=\chi(M)$, the Euler-Poincar\'e characteristic of $M$.
\end{Ex}

Atiyah's idea consists in recalling that the Hilbert spaces associated with geometric operators as above, are also modules over $C^\infty(M)$ (or even $C(M)$) acting by pointwise multiplication on sections. Now a differential operator of course does not commute with multiplication by functions but, because of Leibniz'rule, it commutes with functions modulo operators of lower order. 

\begin{Ex} Take $M=S^1$ and $P=\frac{1}{i}\frac{d}{dt}$ as in Example \ref{circle}. Then for $f,\xi\in C^1(S^1)$:
$$[P,f](\xi)=(Pf-fP)(\xi)=\frac{1}{i}((f\xi)'-f\xi')=\frac{f'}{i}\xi,$$
so that $[P,f]$ is pointwise multiplication by $\frac{f'}{i}$, which extends to a bounded operator on $L^2$, as it involves no differentiation anymore.
\end{Ex}

If $P$ is pseudo-differential of order 0, then actually $[P,f]$ is a compact operator for every $f\in C(M)$.

\begin{Exo} Check this for the operator $\frac{P}{(1+P^2)^{1/2}}$ of Example \ref{circle}.
\end{Exo}

Strongly motivated by Example \ref{deRham}, we define:

\begin{Def} Let $X$ be a compact space. A triple $(\HH^+,\HH^-,F)$ is a {\bf Fredholm module} over $C(X)$ if
\begin{itemize}
\item $\HH^+,\HH^-$ are Hilbert spaces which are modules over $C(X)$, i.e. there exists $*$-homomorphisms $\pi_\pm:C(X)\rightarrow\mathcal{B}(\HH^\pm)$.
\item $F:\HH^+\rightarrow\HH^-$ is a bounded, Fredholm operator.
\item For every $f\in C(X)$, the operator $F\pi_+(f)-\pi_-(f)F$ is compact.
\end{itemize}
\end{Def}

Atiyah denotes by $Ell(X)$ the set of unitary equivalence classes of Fredholm modules over $X$: these are the cycles for $K$-homology. The correspondence $X\rightarrow Ell(X)$ is covariantly functorial: indeed if $\alpha:X\rightarrow Y$ is a continuous map, then $\alpha^*:C(Y)\rightarrow C(X): f\mapsto f\circ\alpha$ denotes the dual $*$-homomorphism, and if $(\HH^+,\HH^-,P)$ is a Fredholm module over $C(X)$, composing with $\alpha^*$ yields the Fredholm module $\alpha_*(\HH^+,\HH^-,P)$ over $C(Y)$.

In his paper \cite{AtiEll}, Atiyah leaves as an open question the precise definition of the equivalence relation to impose on $Ell(X)$ in order to get $K_0(X)$, mentioning that it should be related to homotopy. Kasparov \cite{KasConsp} worked out the precise notion of homotopy (we will expand on this in subsection \ref{KKordinary} below) so we can define:
$$K_0(X)=Ell(X)/\mbox{homotopy}.$$

\subsection{$G$-Fredholm modules}

Here we explain how to modify the construction from the previous subsection when we throw a group action in the picture.

Let $X$ be now a proper $G$-space. Say that $X$ is {\it $G$-compact} if the quotient space $G\backslash X$ is compact.

\begin{Def}\label{GFred} Let $X$ be a proper, $G$-compact space. A {\bf $G$-Fredholm module} over $X$ is a triple $(\HH^+,\HH^-,F)$ where:
\begin{itemize}
\item $\HH^+,\HH^-$ are Hilbert spaces equipped with (strongly continuous) unitary representations of $G$ and $G$-covariant representations $\pi_+,\pi_-$ of $C_0(X)$.
\item $F:\HH^+\rightarrow\HH^-$ is a linear, bounded, $G$-equivariant operator.
\item For every $f\in C_0(X)$, the operator $\pi_-(f)F-F\pi_+(f)$ is compact.
\item There exists a linear, bounded, $G$-equivariant operator $Q:\HH^-\rightarrow\HH^+$ such that for every $f\in C_0(X)$ the operators $\pi_+(f)(QF-1)$ and $\pi_-(f)(FQ-1)$ are compact.
\end{itemize}\end{Def}

For $X$ a proper $G$-compact space, we then define the {\it $G$-equivariant K-homology} of $X$:
$$K_0^G(X)=\{G-\mbox{Fredholm modules over}\;X\}/(G-\mbox{equivariant homotopy}).$$

We can then define the LHS of Conjecture \ref{BCconj}

\begin{Def}\label{LHS} $$K_0^G(\underline{E}G):=\lim_{X\subset\underline{E}G} K_0^G(X)$$
where the limit is taken on the directed set of proper $G$-compact subsets $X$ of $\underline{E}G$.
\end{Def}

The limit on the RHS of Definition \ref{LHS} takes care of the fact that $\underline{E}G$ is not necessarily $G$-compact and compensates for its lack of uniqueness (recall that it is only defined up to $G$-equivariant homotopy).

The next example generalizes Example \ref{deRham} to the $G$-equivariant setting.

\begin{Ex}\label{equivpseudodiff} Let $M$ be a smooth, complete Riemannian manifold, and let $G\curvearrowright M$ be a proper, isometric, $G$-compact action. Let $E^\pm$ be $G$-equivariant smooth vector bundles over $M$, let $\HH^\pm=L^2(M,E^\pm)$ denote their spaces of $L^2$-sections. Let $F:\HH^+\rightarrow \HH^-$ be a $G$-equivariant, elliptic, pseudo-differential operator of order 0. Then $(\HH^+,\HH^-,F)$ is a $G$-Fredholm module over $M$.
\end{Ex}

\begin{Ex}\label{equivdeRham} To get closer to Example \ref{deRham}, let us specialize Example \ref{equivpseudodiff} by taking $\HH^\pm=L^2(\Lambda^\pm TM), P=d+d^*$ and $F=P(1+P^*P)^{-1/2}$. The operator $Q$ in Definition \ref{GFred} can be taken to be $Q=F^*$. The operator $QF-1=-(1+P^*P)^{-1}$ is not compact when $M$ is not compact, but it is ``locally compact'', meaning that $\pi_+(f)(QF-1)$ is compact for $f\in C_0(M)$. To see this latter point, we may assume that $f$ is compactly supported with support in an open chart of $M$, allowing us to replace $M$ by the $n$-torus $\T^n\;(n=\dim M)$. On $\T^n$ the desired compactness can be shown using Fourier series.
\end{Ex}

\section{KK-theory and the assembly map}

References for this section: \cite{Black}, \cite{KasConsp}, \cite{KasInv}, \cite{MV03}.

To define the assembly map, we need Kasparov theory, which actually provides the right framework both for K-theory and K-homology of $C^*$-algebras.

\subsection{Kasparov theory}\label{KKordinary}

The goal is to associate, to every pair $(A,B)$ of separable $C^*$-algebras, abelian groups $KK_i(A,B)\;(i=0,1)$, contravariant in $A$, covariant in $B$, such that
\begin{itemize}
\item The K-theory group $K_i(B)$ is naturally isomorphic to $KK_i(\C,B)$.
\item (stability) If $p\in\mathcal{K}$ is a rank 1 projection, the homomorphism $A\rightarrow A\otimes\mathcal{K}:a\mapsto a\otimes p$ induces an isomorphism $KK_i(A\otimes\mathcal{K},B)\simeq KK_i(A,B)$. Similarly $KK_i(A,B)\simeq KK_i(A,B\otimes\mathcal{K})$.
\item (split exact sequences) In both variables, $KK_i$ maps split short exact sequences of $C^*$-algebras to split short exact sequences of groups.
\end{itemize}

The main property of KK-theory is the existence of an associative product, the {\it Kasparov product}:

\begin{Thm}\label{KKproduct} (Kasparov \cite{KasConsp}) There exists an associative, bi-additive product
$$KK_i(A,B)\times KK_j(B,C)\rightarrow KK_{i+j}(A,C):(x,y)\mapsto x\otimes_B y$$
(with indices counted modulo 2) that turns $KK_0(A,A)$ into a unital ring. Moreover any $*$-homomorphism $f:A\rightarrow B$ defines an element $[f_*]\in KK_0(A,B)$ such that $[f_*]\otimes_B y=f^*(y)$ for any $y\in KK_i(B,C)$. Similarly, if $g:B\rightarrow C$ is a $*$-homomorphism, then $x\otimes_B [g_*]=g_*(x)$ for every $x\in KK_i(A,B)$.
\end{Thm}

This means that the assignment $f\mapsto [f_*]$ defines a functor from the category of separable $C^*$-algebras with $*$-homomorphisms, to the additive category with the same objects and elements of $KK_0(A,B)$ as morphisms from $A$ to $B$. If $f,g:A\rightarrow B$ are $*$-homomorphisms, then $f+g$ is usually not a $*$-homomorphism, but $[f_*]+[g_*]$ makes sense in $KK_0(A,B)$.

\medskip
After this advertisement for $KK$-theory, let us go into more details.

\medskip
We recall some terminology regarding Hilbert $C^*$-modules (for which a good reference is \cite{Lance}). If $B$ is a $C^*$-algebra and $\mathcal{E}$ is a (right) Hilbert $C^*$-module over $B$, we denote by $\mathcal{L}_B(\mathcal{E})$ the $C^*$-algebra of adjointable operators over $\mathcal{E}$, and by $\mathcal{K}_B(\mathcal{E})$ the closed ideal of $B$-compact operators. Recall that an operator is {\it $B$-compact} if it is a norm-limit of sums of operators of the form $x\mapsto\langle x|\xi\rangle\eta$, for some $\xi,\eta\in\mathcal{A}$.

\begin{Def}\label{KK0cycle} A {\bf cycle} for $KK_0(A,B)$ is a triple $(\mathcal{E}^+,\mathcal{E^-},F)$ where
\begin{itemize}
\item $\mathcal{E}^\pm$ is a Hilbert $C^*$-module over $B$, equipped with a representation of $A$ given by a $*$-homomorphism $\pi_\pm:A\ra\mathcal{L}_B(\mathcal{E}^\pm)$, and $F\in\mathcal{L}_B(\mathcal{E}^+,\mathcal{E}^-)$.
\item For every $a\in A$, we have $(F^*F-1)\pi_+(a)\in\mathcal{K}_B(\mathcal{E}^+)$ and $(FF^*-1)\pi_-(a)\in\mathcal{K}_B(\mathcal{E}^-)$.
\item For every $a\in A$, we have $F\pi_+(a)-\pi_-(a)F\in\mathcal{K}_B(\mathcal{E}^+,\mathcal{E}^-)$.
\end{itemize}
\end{Def}

\begin{Ex} If $f:A\ra B$ is a $*$-homomorphism, define $[f_*]=(B,0,0)$, where $B$ is viewed as a right $C^*$-module over itself, and as a left $A$-module through $f$. Then, since $\mathcal{K}_B(B)=B$, we see that all conditions in Definition \ref{KK0cycle} are satisfied.
\end{Ex}

\begin{Def}\label{KK1cycle} A {\bf cycle} for $KK_1(A,B)$ is a pair $(\mathcal{E},P)$ where
\begin{itemize}
\item $\mathcal{E}$ is a Hilbert $C^*$-module over $B$, equipped with a representation of $A$ given by a $*$-homomorphism $\pi:A\ra\mathcal{L}_B(\mathcal{E})$, and $P=P^*\in\mathcal{L}_B(\mathcal{E})$.
\item For every $a\in A$, we have $(P^2-P)\pi(a)\in\mathcal{K}_B(\mathcal{E})$.
\item For every $a\in A$, we have $[P,\pi(a)]\in\mathcal{K}_B(\mathcal{E})$.
\end{itemize}
\end{Def}

The relevant notion of homotopy is then:

\begin{Def}\label{homotopy} Two cycles $x_0,x_1$ for $KK_i(A,B)$ are {\bf homotopic} if there exists a cycle $y$ for $KK_i(A,B\otimes C[0,1])$ such that $e^{(0)}_*(y)$ is isomorphic to $x_0$ and $e^{(1)}_*(y)$ is isomorphic to $x_1$, where $e^{(t)}:B\otimes C[0,1]\ra B$ denotes evaluation at time $t$.
\end{Def}

We then define $KK_i(A,B)$ as the set of corresponding cycles modulo homotopy.

\begin{Prop}\label{KKabelian} Equipped with direct sum, $KK_i(A,B)\;(i=0,1)$ is an abelian group.
\end{Prop}

\bpr Say that a cycle for $KK_i(A,B)$ is {\it degenerate} if, in either Definition \ref{KK0cycle} or Definition \ref{KK1cycle}, all operators which are assumed to be $B$-compact, are actually 0. 

{\it Claim:} Any degenerate cycle for $KK_0(A,B)$ (resp. $KK_1(A,B)$) is homotopic to the trivial cycle $(0,0,0)$. Indeed (say $i=1$), let $x=(\mathcal{E},P)$ be a degenerate cycle for $KK_1(A,B)$. Then
$y=(\mathcal{E}\otimes C_0[0,1[,P\otimes 1)$ is a cycle for $KK_1(A,B\otimes C[0,1])$ (reason: since $\mathcal{K}_B(\mathcal{E}\otimes C_0[0,1[)=\mathcal{K}_B(\mathcal{E})\otimes C_0[0,1[$, we see that $y$ is a cycle because $x$ is degenerate). Since $e^{(t)}_*(y)=x$ for $0\leq t<1$ and $e^{(1)}_*(y)=(0,0)$, we see that $y$ provides the desired homotopy.

We now start the proof of the Proposition itself.
\begin{itemize}
\item ($i=0$) The zero-element of $KK_0(A,B)$ is the class of the trivial cycle $(0,0,0)$. The opposite of the cycle $(\mathcal{E}^+,\mathcal{E}^-,F)$ is $(\mathcal{E}^-,\mathcal{E}^+,F^*)$. Indeed
$$(\mathcal{E}^+\oplus\mathcal{E}^-,\mathcal{E}^-\oplus\mathcal{E}^+,\left(\begin{array}{cc}F\cos(\frac{\pi}{2}t) & -\sin(\frac{\pi}{2}t) \\\sin(\frac{\pi}{2}t) & F^*\cos(\frac{\pi}{2}t)\end{array}\right))\;(t\in[0,1])$$
is a homotopy between $(\mathcal{E}^+\oplus\mathcal{E}^-,\mathcal{E}^-\oplus\mathcal{E}^+,\left(\begin{array}{cc}F & 0\\ 0 & F^*\end{array}\right))$ and the degenerate cycle $(\mathcal{E}^+\oplus\mathcal{E}^-,\mathcal{E}^-\oplus\mathcal{E}^+,\left(\begin{array}{cc}0 & -1 \\1 & 0\end{array}\right))$. We conclude using the claim.

\item ($i=1$) The zero element of $KK_1(A,B)$ is the class of $(0,0)$. The opposite of the cycle $(\mathcal{E},P)$ is $(\mathcal{E},1-P)$. Indeed
$$(\mathcal{E}\oplus\mathcal{E},\left(\begin{array}{cc}P\cos(\frac{\pi}{2}t)+\frac{1}{2}(1-\cos(\frac{\pi}{2}t)) & \frac{1}{2}\sin(\frac{\pi}{2}t) \\\frac{1}{2}\sin(\frac{\pi}{2}t) & -P\cos(\frac{\pi}{2}t)+\frac{1}{2}(1+\cos(\frac{\pi}{2}t))\end{array}\right)\;(t\in[0,1])$$
is a homotopy between $(\mathcal{E}\oplus\mathcal{E},\left(\begin{array}{cc}P & 0 \\0 & 1-P\end{array}\right)$ and the degenerate cycle $\left(\begin{array}{cc}1/2 & 1/2 \\1/2 & 1/2\end{array}\right)$. Again, the claim allows to conclude.\epr
\end{itemize}

\begin{Exo} Check the details of the previous proof.
\end{Exo}

\begin{Ex} (suggested by the referee) The ring structure on $KK_0(A,A)$ is not commutative in general. For example, $KK_0(\C^2,\C^2)$ is isomorphic as a ring to $M_2(\Z)$.
\end{Ex}

\subsection{Equivariant Kasparov theory} 

We now indicate the changes to be made to go from KK-theory to equivariant KK-theory. So let $G$ satisfy the usual assumptions. Let $A,B$ be $G-C^*$-algebras, i.e. $C^*$-algebras endowed with continuous actions of $G$ by $*$-automorphisms. Let $\mathcal{E}$ be a $C^*$-module over $B$, assume that $\mathcal{E}$ is equipped with a continuous action of $G$ by linear operators. Say that $\mathcal{E}$ is a {\it Hilbert $G-C^*$-module} if, for every $g\in G,b\in B,\xi,\eta\in\mathcal{E}$:
$$g(\xi.b)=(g\xi).(gb)\;\mbox{and}\;\langle g\xi| g\eta\rangle =g.\langle\xi|\eta\rangle.$$

Then, in the definition of cycles for $KK_i^G(A,B)$, we require $C^*$-modules to be $G-C^*$-modules over $B$, and representations of $A$ to be $G$-equivariant. Moreover (say for $i=1$) the operator $P$ should satisfy the following two conditions:
\begin{itemize}
\item The map $g\mapsto g\cdot P$ is {\it norm} continuous on $G$.
\item For every $a\in A,g\in G$, we have $(g\cdot P-P)\pi(a)\in\mathcal{K}_B(\mathcal{E})$.
\end{itemize}

The equivariant KK-groups $KK_i^G(A,B)$ are then defined as cycles modulo homotopy, as in subsection \ref{KKordinary}. These are abelian groups, as in Proposition \ref{KKabelian}. Moreover, assuming the $G-C^*$-algebras to be separable, all desirable properties listed in the beginning of subsection \ref{KKordinary} hold, in particular the equivariant analogue of Theorem \ref{KKproduct} on the Kasparov product (see \cite{KasInv}).

In particular $KK_0^G(A,A)$ is a unital ring, with unit given by $1_A=[id_A,*]$ where $id_A$ is the identity on $A$; this is the class represented by $(A,0,0)$ with the usual left action of $A$.

\begin{Ex}\label{compactgroup} For $G$ compact and $A=\C$ with trivial $G$-action, we have $KK_0^G(\C,\C)=R(G)$, the representation ring of isomorphism classes of finite-dimensional representations of $G$. On the other hand $KK_1^G(\C,\C)=0$.
\end{Ex}

A useful construction in equivariant KK-theory is {\it extension of scalars}. For $C^*$-algebras $A,B$, denote by $A\otimes_{min} B$ the {\it minimal} $C^*$-tensor product of $A$ and $B$.

\begin{Prop}\label{extension} Let $A,B,C$ be separable $G-C^*$-algebras. There exists a homomorphism 
$\tau_C:KK_i^G(A,B)\ra KK_i^G(A\otimes_{min}C,B\otimes_{min}C)$ which is natural and compatible with the Kasparov product. For $i=0$ it is given by $(\mathcal{E}^+,\mathcal{E}^-,F)\mapsto(\mathcal{E}^+\otimes C,\mathcal{E}^-\otimes C,F\otimes 1_C)$, and analogously for $i=1$.
\end{Prop}

We may now revisit the definition of {\it equivariant K-homology} (compare with Definition \ref{LHS}):

\begin{Def} If $X$ is a proper, $G$-compact space, set
$$K_i^G(X):=KK_i^G(C_0(X),\C)\;(i=0,1).$$
If $Y$ is any proper $G$-space (e.g. $Y=\underline{E}G$), set
$$K_i^G(Y)=\lim_X K_i^G(X)\;(i=0,1),$$
where the limit is taken on the directed set of proper $G$-compact subsets $X$ of $Y$.
\end{Def}

We now pass to the compatibilities between equivariant and non-equivariant KK-theory. Recall that, if the group $G$ acts continuously on the $C^*$-algebra $A$ by $*$-automorphims, we may define the {\it full crossed product} $A\rtimes G$ and the {\it reduced crossed product} $A\rtimes_r G$. Both are defined as suitable completions of $C_c(G,A)$, the space of compactly supported continuous $A$-valued functions on $G$ equipped with {\it twisted convolution}: for $f_1,f_2\in C_c(G,A),g\in G$:
\begin{equation}\label{twisted}
(f_1\ast f_2)(g):=\int_G f_1(h)(h\cdot (f_2(h^{-1}g)))\,dh
\end{equation}

\begin{Ex} For $A=\C$ with the trivial $G$-action, we get $\C\rtimes G= C^*(G)$, the {\bf full group $C^*$-algebra}; and $\C\rtimes_r G=C^*_r(G)$, the reduced group $C^*$-algebra appearing in Conjecture \ref{BCconj}. 
\end{Ex}

We denote by $\lambda_{A,G}:A\rtimes G\ra A\rtimes_r G$ the canonical $*$-epimorphism from the full to the reduced crossed product. If $G$ is amenable, then $\lambda_{A,G}$ is an isomorphism. The converse is false. For example, if $G$ is arbitrary and $X$ is a proper $G$-space, we will use later that $\lambda_{C_0(X),G}$ is always an isomorphism (see e.g. \cite{Anan}).

For $A=\C$, we denote by $\lambda_G:C^*(G)\ra C^*_r(G)$ the canonical $*$-epimorphism associated with the regular representation of $G$. Now $\lambda_G$ is an isomorphism if and only if $G$ is amenable.

\begin{Prop} There exist natural {\bf descent} homomorphisms $$\widetilde{j_G}:KK_i^G(A,B)\ra KK_i(A\rtimes G,B\rtimes G)$$ and $$j_G:KK_i^G(A,B)\ra KK_i(A\rtimes_r G,B\rtimes_rG),$$
compatible with Kasparov products, such that the following diagram is commutative:
$$\left.\begin{array}{ccc}KK_i^G(A,B) & \stackrel{\widetilde{j_G}}{\longrightarrow} & KK_i(A\rtimes G,B\rtimes G) \\ & & \\j_G\downarrow&  & \downarrow(\lambda_{B,G})_* \\ & & \\ KK_i(A\rtimes_rG,B\rtimes_rG) & \stackrel{(\lambda_{A,G})^*}{\longrightarrow} & KK_i(A\rtimes G,B\rtimes_r G)\end{array}\right.$$
\end{Prop}

\subsection{The map $\tilde{\mu}_G$}

We are now almost ready to define the last ingredient in Conjecture \ref{BCconj}, namely the assembly map. It will be obtained by first applying descent to an equivariant $K$-homology class, and then pairing the resulting $KK$-class with a canonical class associated with the proper $G$-space.

\subsubsection{The canonical idempotent}

Let $X$ be a proper, $G$-compact space. Denote by $C_c(X)$ the space of compactly supported, continuous functions on $G$. 

\begin{Lem}\label{canon}\begin{enumerate}
\item There exists a non-negative $h\in C_c(X)$ such that $\int_G h(s^{-1}\cdot x)\,ds=1$ for every $x\in X$.
\item Set $p_X(x,s):=\sqrt{h(x)h(s^{-1}\cdot x)}$. Then $p_X$ is an idempotent in $C_c(X\times G)$ (endowed with twisted convolution as in (\ref{twisted}).
\end{enumerate}
\end{Lem}

\bpr \begin{enumerate} 
\item By $G$-compactness, we find a non-negative function $g\in C_c(X)$ whose support meets every $G$-orbit in $X$. Then we set $h(x):=\frac{g(x)}{\int_G g(t^{-1}\cdot x)\,dt}$.

\begin{Ex} For $G=\Z$ acting by integer translations on $X=\R$, we may take $$h(x)=\left\{\begin{array}{cl}x+1 & \mbox{if}\;x\in[-1,0] \\1-x & \mbox{if}\;x\in[0,1] \\0 & \mbox{otherwise}\end{array}\right. .$$\end{Ex}

\item For $f_1,f_2\in C_c(X\times G),x\in X, g\in G$, we have by (\ref{twisted}):
$$(f_1\ast f_2)(x,g)=\int_G f_1(x,s)f_2(s^{-1}\cdot x,s^{-1}g)\,ds.$$
Then
$$p_X^2(x,g)=\int_G [h(x)h(s^{-1}\cdot x)h(s^{-1}\cdot x)h(g^{-1}\cdot x)]^{1/2}\,ds=(h(x)h(g^{-1}\cdot x))^{1/2}=p_X(x,g)$$
where the 2nd equality follows from the first part of the lemma.
\epr
\end{enumerate}

We denote by $[p_X]$ the class of the projection $p_X$ in $K_0(C_0(X)\rtimes G)= KK_0(\C,C_0(X)\rtimes G)$. Note that $[p_X]$ is independent of the choice of the function $h$ in Lemma \ref{canon}: the reason is that the set of functions $h$'s satisfying part 1 of the lemma, is convex. 

\begin{Def} $[p_X]$ is the {\bf canonical element} in $KK_0(\C,C_0(X)\rtimes G)$.
\end{Def}

\subsubsection{The Baum-Connes conjecture with coefficients}

For $X$ a proper, $G$-compact subset of $\underline{E}G$, we define the map
$$\tilde{\mu}_{G,X}: KK_i^G(C_0(X),\C)\ra KK_i(\C,C^*(G))=K_i(C^*G):x\mapsto [p_X]\otimes_{C_0(X)\rtimes G}\;\tilde{j}_G(x).$$

Note that here we do not need the full force of the Kasparov product. Indeed viewing $p_X$ a a $*$-homomorphism $\C\ra C_0(X)\rtimes G$, we may as well define $\tilde{\mu}_{G,X}(x):=p_X^*(\tilde{j}_G(x))$.

For $X$'s in the directed set of proper $G$-compact subsets of $\underline{E}G$, the maps $\tilde{\mu}_{G,X}$ match and define 
$$\tilde{\mu}_G:K_i^G(\underline{E}G)=\lim_XKK_i^G(C_0(X),\C)\ra K_i(C^*G).$$

More generally, for any $C^*$-algebra $B$ with a continuous $G$-action, we set
$$\tilde{\mu}_{G,B}:KK_i^G(\underline{E}G,B):=\lim_X KK_i^G(C_0(X),B)\ra K_i(B\rtimes G):x\mapsto [p_X]\otimes_{C_0(X)\rtimes G}\;j_{G,B}(x)$$
(where $x\in KK_i^G(C_0(X),B)$).

\begin{Def} For $B$ as above, the {\bf assembly map for $G$ and $B$} is the map $\mu_{G,B}:KK_i^G(\underline{E}G,B)\ra K_i(B\rtimes_r G)$ given by $\mu_{G,B}:=(\lambda_{G,B})_*\circ\tilde{\mu}_{G,B}$.

\medskip
The {\bf assembly map} for $G$ is $\mu_G=(\lambda_G)_*\circ\tilde{\mu}_G:K_i^G(\underline{E}G)\ra K_i(C^*_r(G))$.
\end{Def}

The {\it Baum-Connes conjecture with coefficients} for $G$ is the following statement.

\begin{Conj}\label{BCcoeff} For every $C^*$-algebra $B$ equipped with a continuous $G$-action by $*$-automorphisms, the map $\mu_{G,B}:KK_i^G(\underline{E}G,B)\ra K_i(B\rtimes_r G)\;(i=0,1)$ is an isomorphism.
\end{Conj}

Conjecture \ref{BCconj} is of course the case $B=\C$ of Conjecture \ref{BCcoeff}. The latter has pros and cons with respect to the former.

\begin{itemize}
\item Conjecture \ref{BCcoeff} is known to be stable under passing to a closed subgroup, while the corresponding statement is unknown for Conjecture \ref{BCconj}.
\item Conjecture \ref{BCcoeff} was shown to hold for important classes of groups, see Theorems \ref{gamma=1} and \ref{gammaexists} below.
\item In spite of those positive results, Conjecture \ref{BCcoeff} is false in general: counterexamples have been constructed by N. Higson, V. Lafforgue and G. Skandalis \cite{HLS}.
\end{itemize}

\subsection{The assembly map for discrete groups}

In this whole subsection, assume groups to be discrete. 

\subsubsection{The strong Novikov conjecture}

If $X$ is $G$-proper and moreover $G$ acts {\bf freely} on $X$, then $K_i^G(X)=K_i(G\backslash X)$. In particular $K_i^G(EG)=K_i(BG)$. By the universal property of $\underline{E}G$, we have a $G$-equivariant map $EG\ra\underline{E}G$, unique up to $G$-equivariant homotopy. So we get a canonical map $\iota_*:K_i(BG)\ra K_i^G(\underline{E}G)$ which is known to be rationally injective (see section 7 in \cite{BCH94}).

The following statement is known as the {\bf strong Novikov conjecture} for $G$:

\begin{Conj}\label{SNC} The map $\mu_G\circ\iota_*:K_i(BG)\ra K_i(C^*_r(G))\;(i=0,1)$ is rationally injective.
\end{Conj}

\subsubsection{Functoriality of assembly}

Here we discuss the functoriality properties of the assembly maps $\tilde{\mu}_G$ and $\mu_G$. The full $C^*$-algebra is fully functorial: every group homomorphism $\alpha:G\ra H$ induces a $*$-homomorphism denoted by the same letter: $\alpha:C^*(G)\ra C^*(H)$. It turns out that the maps $\tilde{\mu}$ are fully functorial (see \cite{MV03}), i.e. the following diagram commutes:
$$\left.\begin{array}{ccc}K_i^G(\underline{E}G) & \stackrel{\tilde{\mu}_G}{\longrightarrow} & K_i(C^*(G)) \\ & & \\ \alpha_*\downarrow&  & \downarrow \alpha_* \\ & & \\ K_i(\underline{E}H) & \stackrel{\tilde{\mu}_H}{\longrightarrow} & K_i(C^*(H))\end{array}\right.$$

It is not known whether $\mu_G$ is functorial. However, since the reduced group $C^*$-algebra is functorial for injective group homomorphisms (every injective homomorphism $G\hookrightarrow H$ induces an injective $*$-homomorphism $C^*_r(G) \hookrightarrow C^*_r(H)$), the map $\mu_G$ is certainly functorial with respect to injective group homomorphisms.

Of course, on the class of amenable groups $\mu_G$ is fully functorial as $C^*(G)=C^*_r(G)$. However, for non-amenable groups $C^*_r(G)$ is often simple (this is the case for free groups, or more generally for lattices in simple Lie groups, see \cite{BeCoHa}), so $C^*_r(G)$ is generally not functorial for non-injective group homomorphisms.

\subsubsection{The obstruction of property (T)}\label{prop(T)}

The good functorial properties of $\tilde{\mu}_G$ raise a natural question: could it be that $\tilde{\mu}_G$ is an isomorphism? If yes, because $\mu_G=(\lambda_G)_*\circ\tilde{\mu}_G$ and $\mu_G$ should be an isomorphism by Conjecture \ref{BCconj}, this would imply that $(\lambda_G)_*:K_i(C^*(G))\ra K_i(C^*_r(G))$ is an isomorphism. However, this is not true in general, as was observed by Connes back in 1980:

\begin{Prop}\label{(T)} If the non-compact, locally compact group $G$ has Kazhdan's property (T), then $(\lambda_G)_*:K_0(C^*(G))\ra K_0(C^*_r(G))$ is not injective.
\end{Prop}

V. Lafforgue proved in \cite{Laff} that $\mu_G$ is an isomorphism for (countable) hyperbolic groups, some of which have property (T) (e.g. uniform lattices in the Lie group $Sp(n,1),\,(n\geq 2)$). For those, it follows from Proposition \ref{(T)} that $\tilde{\mu}_G$ is NOT onto.

\medskip
{\bf Proof of Proposition \ref{(T)}:} The group $G$ has property (T) if and only if there exists a projection $p_G\in C^*(G)$ such that, for every unitary representation $\pi$ of $G$, the projection $\pi(p_G)$ is the orthogonal projection of the subspace of $\pi(G)$-fixed vectors. Then
\begin{itemize}
\item Taking for $\pi$ the trivial one-dimensional representation $\epsilon$ of $G$, we get $\epsilon(p_G)=1$ which shows that $[p_G]$ has infinite order in $K_0(C^*(G))$.
\item Taking for $\pi$ the regular representation $\lambda_G$, which has no non-zero $G$-fixed vectors as $G$ is non-compact, we get $\lambda_G(p_G)=0$, completing the proof. $\blacksquare$
\end{itemize}

In a paper in the same series, P. Julg \cite{JulgK-am} discusses when $(\lambda_G)_*: K_i(C^*(G))\ra K_i(C^*_r(G))$ is an isomorphism, leading to the concept of {\it K-amenability}.

\section{The ``Dirac-dual Dirac'' method}

The ``Dirac-dual Dirac'' method, introduced by Kasparov \cite{KasConsp} in special cases and formalized by Tu \cite{Tu00}, is one of the most powerful tools to attack Conjectures \ref{BCconj} and \ref{BCcoeff}.

\subsection{The abstract method}

We need the concept of {\it proper} $G-C^*$-algebra $D$. Intuitively, this means that if we identify the center $Z(D)$ with $C_0(X)$ for some locally compact space $X$, the $G$-action on $X$ is proper. More precisely:

\begin{Def} Let $X$ be a $G$-space. A {\bf $G-X-C^*$-algebra} is a $G-C^*$-algebra $D$ equipped with a $G$-equivariant homomorphism $C_0(X)\ra Z(M(D))$ (the center of
the multiplier algebra of $D$) which moreover is non-degenerate, i.e $\overline{span(C_0(X)D)}=D$. The $G-C^*$-algebra $D$ is then {\bf proper} if there exists a proper $G$-space $X$ such that $D$ is a $G-X-C^*$-algebra.
\end{Def}

It is a crucial result due to Chabert, Echterhoff and Meyer \cite{CEM}, that the Baum-Connes conjecture with proper coefficients holds for arbitrary groups\footnote{For $D$ commutative and $G$ a connected Lie group, this result was obtained previously in \cite{Val}.}.

\begin{Thm}\label{Tuproper} If $D$ is a proper $G-C^*$-algebra, then the map $\mu_{G,D}:KK_i^G(\underline{E}G,D)\ra K_i(D\rtimes_r G)\;(i=0,1)$ is an isomorphism.
\end{Thm}

Consider now the following assumption on a given group $G$

\medskip
{\bf Assumption (*):} There exists a proper $G-C^*$-algebra $D$ and elements $\alpha\in KK_i^G(D,\C)$ (``the Dirac element'') and $\beta\in KK_i^G(\C,D)$ (``the dual-Dirac element'') such that $\beta\otimes_D\alpha= 1_\C$ in $KK_0^G(\C,\C)$.

\medskip 
Assumption (*) is actually very strong, as it implies the Baum-Connes conjecture with coefficients! Indeed:

\begin{Prop}\label{Strongass} Under Assumption (*), Conjecture \ref{BCcoeff} holds.
\end{Prop}

\bpr Let $B$ be a $G-C^*$-algebra. Write $B\otimes D$ for $B\otimes_{min} D$. We have the following commutative diagram:
$$\begin{array}{ccccc}KK_i^G(\underline{E}G,B) & \stackrel{\otimes_B \tau_B(\beta)}{\longrightarrow} & KK_i^G(\underline{E}G,B\otimes D) & \stackrel{\otimes_{B\otimes D}\tau_B(\alpha)}{\longrightarrow} & KK_i^G(\underline{E}G,B) 
\\ \downarrow\mu_{G,B} &  &\simeq\downarrow\mu_{G,B\otimes D}  &  & \downarrow\mu_{G,B}  
\\K_i(B\rtimes_r G) & \stackrel{\otimes_{B\rtimes_r G}j_G(\tau_B(\beta))}{\longrightarrow} & K_i((B\otimes D)\rtimes_rG) & \stackrel{\otimes_{(B\otimes D)\rtimes_rG}j_G(\tau_B(\alpha))}{\longrightarrow} & K_i(B\rtimes_r G)
\end{array}$$
The central vertical map is an isomorphism as $B\otimes D$ is a proper $G-C^*$-algebra, so Proposition \ref{Tuproper} applies. Now Assumption (*) implies that compositions of horizontal maps are isomorphisms, at the top and bottom line. Injectivity of $\mu_{G,B}$ then follows from the left hand part of the diagram, while surjectivity follows from the right hand part. \epr

\medskip
Take $B=\C$ in the previous proof and observe that the same proof would work with $\mu_G$ replaced by $\tilde{\mu}_G$, so that Assumption (*) also implies that $\tilde{\mu}_G$ is an isomorphism. From the discussion in subsection \ref{prop(T)}, this implies that $(\lambda_G)_*$ is an isomorphism, which does not hold e.g. for property (T) groups (see Proposition \ref{(T)}). This shows that Assumption (*), although satisfied in some interesting cases (see Theorem \ref{gamma=1} below) is actually too strong. J.-L. Tu \cite{Tu00} came up with the following substitute.

\begin{Def}\label{gamma} A {\bf $\gamma$-element} is an element $\gamma\in KK_0^G(\C,\C)$ such that
\begin{enumerate}
\item There exists a proper $G-C^*$-algebra $D$, and elements $\alpha\in KK_i^G(D,\C)$ (``Dirac element''), $\beta\in KK_i^G(\C,D)$ (``dual-Dirac element'') such that $\gamma=\beta\otimes_D \alpha$.
\item For every compact subgroup $F$ of $G$, the image of $\gamma$ under the restriction map $Rest_G^F:KK_0^G(\C,\C)\ra R(F)$ is the trivial one-dimensional representation $\epsilon_F$.
\end{enumerate}\end{Def}

The second condition in the above definition guarantees the non-triviality of a $\gamma$-element. Assumption (*) means that the unit $1_\C$ of $KK_0^G(\C,\C)$ is a $\gamma$-element. The following result is due to Tu.

\begin{Thm}\label{Tu}\begin{enumerate}
\item If a $\gamma$-element exists in $KK_0^G(\C,\C)$, then it is unique and moreover it is an idempotent in $KK_0^G(\C,\C)$.
\item Assume that a $\gamma$-element exists. Then
\begin{enumerate}
\item For any $G-C^*$-algebra $B$, the map $\mu_{G,B}$ is injective. In particular the strong Novikov conjecture (Conjecture \ref{SNC}) holds. 
\item For a $G-C^*$-algebra $B$, the map $\mu_{G,B}$ is surjective if and only if the map $\otimes_{B\otimes_r G} \tau_{B\rtimes_r G}(\gamma)$ is the identity on $K_i(B\rtimes_r G)\;(i=0,1)$.
\end{enumerate}\end{enumerate}\end{Thm}

Some comments on the proof. If we have a $\gamma$-element $\gamma$, then $\gamma\otimes_\C \gamma$ is still a $\gamma$-element, so $\gamma\otimes_\C \gamma=\gamma$ follows from uniqueness of $\gamma$. For the second part, consider the same commutative diagram as in the proof of Proposition \ref{Strongass}. It is shown that the composition of the maps on the top row is the identity, hence injectivity of $\mu_{G,B}$. The criterion for surjectivity follows similarly.

\subsection{Examples and results}

Let us give the classical example of a Dirac element.

\begin{Ex}\label{Dirac} Let $M$ be a smooth, complete Riemannian manifold, and let $G\curvearrowright M$ be a proper, isometric, $G$-compact action. Take $\HH^\pm=L^2(\Lambda^\pm TM), P=d+d^*$ and $F=P(1+P^*P)^{-1/2}$. Take moreover $D=C_0(M)$, a proper $G-C^*$-algebra. By Example \ref{equivdeRham} we have the element $\alpha=(\HH^+,\HH^-,F)\in KK_0(C_0(M),\C)$, that we call the Dirac element associated with $M$.
\end{Ex}

Here is now the classical example of a dual-Dirac element.

\begin{Ex}\label{DualDirac} In the situation of Example \ref{Dirac}, assume moreover that $M$ is simply connected with non-positive curvature (so that $M$ is a model for $\underline{E}G$). By the Cartan-Hadamard theorem, any two points $x,y\in M$ are connected by a unique geodesic segment that we denote by $[x,y]$. Let $x_0$ be a base-point in $M$, set $\rho(x)=\frac{d_M(x_0,x)^2}{2}$, where $d_M$ denotes the Riemannian distance on $M$. Denote by $C_0(\Lambda^\pm TM)$ the set of continuous differential forms vanishing at infinity on $M$, and view it as a Hilbert $C^*$-module over $C_0(M)$. View the differential $d\rho$ as a 1-form and define the (unbounded) operator $P:C_0(\Lambda^+TM)\ra C_0(\Lambda^-TM)$ as
$$P=ext(d\rho) + int(d\rho),$$
where (for $\omega\in C_0(\Lambda^+ TM)$):  $ext(d\rho)(\omega)=d\rho\wedge\omega$ and $int(d\rho)(\omega)$ is the contraction of the form $\omega$ on the gradient field $\nabla\rho$ of $\rho$. Then $P^*P$ is just pointwise multiplication by the function $\|d\rho\|^2=\|\nabla\rho\|^2$. But $\nabla\rho(x)$ is a tangent vector at $x$ to the geodesic $[x_0,x]$, pointing outwards from $x_0$, of norm $d_M(x_0,x)$. Hence $(1+P^*P)^{-1}$ is pointwise mutiplication by $\frac{1}{1+d_M(x_0,\cdot)^2}$ which is continuous and goes to 0 at infinity of $M$ and therefore defines a compact operator in the sense of $C^*$-modules.

For $g\in G$, the operator $gPg^{-1}$ is the similar operator defined with respect to the base-point $x'_0=gx_0$, using the function $\rho'(x)=\frac{d_M(x_0',x)^2}{2}$.

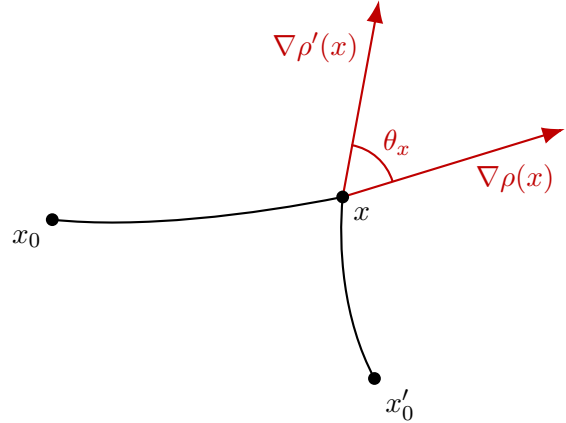
\begin{figure}[h]
\centering

\begin{tikzpicture}[
    scale=1.2,
    point/.style={circle, fill=black, inner sep=1.7pt},
    blackcurve/.style={black, thick},
    redarrow/.style={red!75!black, thick, -{Latex[length=3mm]}},
    every node/.style={font=\small}
]

\coordinate (x)  at (0,0);
\coordinate (x0) at (-3.2,-0.25);
\coordinate (xp0) at (0.35,-2.0);

\coordinate (g)  at (2.45,0.75);
\coordinate (gp) at (0.4,2.17);

\draw[blackcurve]
    (x0) .. controls (-2.25,-0.35) and (-1.0,-0.20) .. (x);

\draw[blackcurve]
    (xp0) .. controls (0.00,-1.35) and (-0.05,-0.55) .. (x);

\draw[redarrow] (x) -- (g);
\draw[redarrow] (x) -- (gp);

\draw[red!75!black, thick]
    (0.55,0.17) arc[start angle=17, end angle=80, radius=0.58];

\node[red!75!black] at (0.6,0.6) {$\theta_x$};

\node[point] at (x0) {};
\node[point] at (xp0) {};
\node[point] at (x) {};

\node[below left] at (x0) {$x_0$};
\node[below right] at (xp0) {$x_0'$};
\node[below right] at (x) {$x$};

\node[red!75!black] at (-0.3,1.65) {$\nabla \rho'(x)$};
\node[red!75!black] at (1.9,0.2) {$\nabla \rho(x)$};

\end{tikzpicture}

\caption{$P$ and $gPg^{-1}$}
\end{figure}

One checks that $\|\nabla(\rho-\rho')\|$ stays bounded on $M$: to see this, use the fact that the angle $\theta_x$ between $\nabla\rho(x)$ and $\nabla\rho'(x)$, goes to 0 when $x$ goes to infinity in $M$). As a consequence, defining $F=P(1+P^*P)^{-1/2}$, we may define the element $\beta=(C_0(\Lambda^+TM),C_0(\Lambda^- TM),F)\in KK_0^G(\C,C_0(M))$ and we call it the dual-Dirac element.

This holds for example if $M$ is either a (finite-dimensional) Euclidean space of a Riemannian symmetric space of the non-compact type. 
\end{Ex}

If $M$ satisfies the assumptions of Examples \ref{Dirac} and \ref{DualDirac}, we may form 
$$\gamma=\beta\otimes_{C_0(M)} \alpha \in KK_0^G(\C,\C):$$
this is a $\gamma$-element (see \cite{KasConsp}) described, in the unbounded operator picture, by $(L^2(\Lambda^+TM),L^2(\Lambda^-TM),D)$ with
\begin{equation}\label{DdualD}
D=d+d^*+ext(d\rho)+int(d\rho).
\end{equation}

\begin{Rem} When $M$ is the $n$-dimensional Euclidean space $\mathbb{E}^n$, the operator $D^2$ is very much related to the harmonic oscillator operator on $\mathbb{E}^n$. Indeed, if $\omega$ is a $k$-form, we have (for $x\in\mathbb{E}^n$):
$$D^2\omega(x)=\Delta\omega(x)+\|x\|^2\omega(x)+ (2k-n)\omega(x).$$
\end{Rem}

\begin{Thm}\label{gamma=1} In the setting of Examples \ref{Dirac} and \ref{DualDirac}, the equality $\gamma=1_{\C}$ holds when:
\begin{itemize}
\item $M=\mathbb{E}^n$ (see Kasparov \cite{KasConsp}).
\item $M=\mathbb{H}^n(\R)$, the $n$-dimensional real hyperbolic space (see Kasparov \cite{KasLor}).
\item $M=\mathbb{H}^n(\C)$, the $n$-dimensional complex hyperbolic space (see Julg-Kasparov \cite{JulKas}).
\end{itemize}
\end{Thm}

So Conjecture \ref{BCcoeff} holds for all groups acting properly isometrically on $M$ as in Theorem \ref{gamma=1}. Now $\gamma$-elements have been constructed in other situations.

\begin{Thm}\label{gammaexists} A $\gamma$-element exists for the following classes of groups.
\begin{itemize}
\item Groups acting properly on locally finite trees; here $\gamma=1_{\C}$ (see Pimsner \cite{Pimsner}).
\item Groups acting properly on locally finite Euclidean buildings (see Kasparov-Skandalis \cite{KaSk1}).
\item Hyperbolic groups (see Kasparov-Skandalis \cite{KaSk2}).
\item Groups with the Haagerup property (in particular amenable groups); here $\gamma=1_{\C}$ (see Higson-Kasparov \cite{HiKa}).
\end{itemize}\end{Thm}

A comment on the case of groups $G$ acting properly on trees (that includes free groups, $SL_2(\Z)$ or $SL_2(\Q_p)$): in \cite{JuVa}, P. Julg and the author constructed a non-trivial element $\gamma\in KK_0^G(\C,\C)$ and showed $\gamma=1_{\C}$, which was enough to guarantee that $G$ is {\it K-amenable}, meaning that for $G-C^*$-algebra $A$, the canonical map $\lambda_{A,G}:A\rtimes G\ra A\rtimes_r G$, induces isomorphisms in K-theory\footnote{For more on K-amenability, see P. Julg's paper \cite{JulgK-am} in the same series.}. Denoting by $V$ the set of vertices in the tree, Pimsner \cite{Pimsner} then constructs elements $\alpha\in KK_0^G(C_0(V),\C),\;\beta\in KK_0^G(\C,C_0(V))$ and identifies the Kasparov product $\beta\otimes_{C_0(V)}\alpha$ with the element $\gamma$ from \cite{JuVa}.

\section{Inverting the assembly map}

\subsection{Motivation}

In 2018, S. Nishikawa, then a PhD student, posted a very interesting paper \cite{Nish} where, based on a lucid analysis of the Dirac-dual Dirac method, he introduced novel ideas for inverting the assembly map, which can in principle serve as an alternative to the Dirac-dual Dirac method for proving Conjectures \ref{BCconj} or \ref{BCcoeff}.

\subsection{Property $(\gamma)$}

In this section, we shall always assume that there exists a $G$-compact model for $\underline{E}G$, that we shall denote by $\underline{E}$. Nishikawa's basic observation (reminiscent of Atiyah's observation in sub-section \ref{AtiyahEll}) is that, if $G$ admits a $\gamma$-element in the sense of Definition \ref{gamma}, then the Fredholm module underlying $\beta\otimes_D\alpha$ carries an extra structure of $C_0(\underline{E})$-module. He then defines:

\begin{Def}\label{(gamma)} Let $x=(\HH^+,\HH^-,F)$ be a cycle for $KK_0^G(\C,\C)$. We say that $x$ has {\bf property $(\gamma)$} if $\HH^\pm$ carries a non-degenerate, $G$-equivariant representation $\pi_\pm$ of $C_0(\underline{E})$ such that:
\begin{enumerate}
\item For every $f\in C_0(\underline{E})$, the function $g\mapsto F\pi_+(g\cdot f)-\pi_-(g\cdot f)F$ is in $C_0(G,\mathcal{K})$.
\item For every compact subgroup $F$ of $G$, the image of $\gamma$ under the restriction map $Rest_G^F:KK_0^G(\C,\C)\ra R(F)$ is the trivial one-dimensional representation $\epsilon_F$.
\end{enumerate}\end{Def}

Here $(g\cdot f)(x)=f(g^{-1}x)$. Recall that a representation $\pi$ of a $C^*$-algebra $A$ on a Hilbert space $\HH$ is said to be {\it non-degenerate} if the linear span of $\pi(A)\HH$ is dense in $\HH$. Observe that condition (2) in Definition \ref{(gamma)} is exactly the same as condition (2) in the definition of a $\gamma$-element (Definition \ref{gamma}). By way of contrast, Condition 1 in Definition \ref{(gamma)} is comparatively less restrictive as it does not require a factorization of the form $\beta\otimes_D\alpha$. Notice that property $(\gamma)$ is really a property of a cycle, not of its homotopy class.

\begin{Lem}\label{nondeg} Let $\pi$ be a non-degenerate representation of $C_0(\underline{E})$ on $\HH$.Then for every $f\in C_0(\underline{E}),T\in\mathcal{K}(\HH)$ we have $\lim_{g\ra\infty}\|\pi(g\cdot f)T\|=0$.
\end{Lem}

\bpr The proof is by a series of reductions:
\begin{itemize}
\item By linearity and density, we may assume that $T$ has rank 1.
\item It is then enough to show $\lim_{g\ra\infty}\|\pi(g\cdot f)(\xi)\|=0$ for every $\xi\in\HH$.
\item Since $\pi$ is non-degenerate, again by linearity and density we may assume that $\xi=\pi(f_0)\xi_0$ for some $f_0\in C_0(\underline{E}),\xi_0\in\HH$, and we must show $$\lim_{g\ra\infty}\|\pi((g\cdot f)f_0)(\xi_0)\|=0.$$
\item By density, we may assume that $f$ and $f_0$ have compact support. Since $G\curvearrowright\underline{E}$ is proper, the function $g\mapsto (g\cdot f)f_0$ has compact support on $G$, which concludes the proof.\epr
\end{itemize}

The basic observation is now that, if $x=(\HH^+,\HH^-,F)$ is a cycle for $KK_0^G(\C,\C)$ with property $(\gamma)$, then for any $G-C^*$-algebra $B$, it is possible to induce up $x$ to an element $\tilde{x}_B\in KK_0^G(C_0(\underline{E})\otimes(B\rtimes_rG),B)$, where $G$ acts trivially on $B\rtimes_rG$. This element $\tilde{x}_B$ can then be used to define a homomorphism $\nu_B^x:K_i(B\rtimes_r G)\ra KK_i^G(C_0(\underline{E}),B)\;(i=0,1)$, so in the direction opposite to the assembly map $\mu_B$. To define $\nu_B^x$ we start by identifying $K_i(B\rtimes_r G)$ with $KK_i^G(\C,B\rtimes_r G)$, with $G$ acting trivially on $\C$ and $B\rtimes_rG$. Then we define
$$\nu_B^x:KK_i^G(\C,B\rtimes_r G)\ra KK_i^G(C_0(\underline{E}),B):y\mapsto \tau_{C_0(\underline{E})}(y)\otimes_{C_0(\underline{E})\otimes(B\rtimes_rG)}\tilde{x}_B.$$
It turns out that $\nu_B^x$ will provide at least a left inverse for $\mu_B$ (see Theorem \ref{Shin} below).

Let us sketch the construction of $\tilde{x}_B$ for $B=\C$ and $G$ discrete. The Hilbert spaces underlying $\tilde{x}_\C\in KK_0^G(C_0(\underline{E})\otimes C^*_r(G),\C)$ will be $\tilde{\HH}^\pm:=\ell^2(G,\HH^\pm)=\ell^2(G)\otimes\HH^\pm$, endowed with the unitary representation of $G$ defined as the tensor product of $\lambda_G$ with the given $G$--representation on $\HH^\pm$.



We need a representation of $C^*_r(G)$ on $\ell^2(G,\HH^\pm)$, commuting with $\pi_\pm(C_0(\underline{E}))$ so that it extends to the tensor product $C_0(\underline{E})\otimes C^*_r(G)$: for this we simply let $C^*_r(G)$ act on $\ell^2(G,\HH^\pm)$ through the {\it right} regular representation of $G$. We denote by $\tilde{\pi}_\pm$ the corresponding representation of $C_0(\underline{E})\otimes C^*_r(G)$. 

\smallskip
Finally we define the operator $\tilde{F}:\tilde{\HH}^+\ra\tilde{\HH}^-$ by
$$(\tilde{F}\xi)(g)=(g\cdot F)(\xi(g)),$$
for $\xi\in\ell^2(G,\HH^+),g\in G$.

We claim that $\tilde{x}_\C :=(\tilde{\HH}^+,\tilde{\HH}^-,\tilde{F})$ defines an element in $KK_0^G(C_0(\underline{E})\otimes C^*_r(G),\C)$. Let us check the 3rd condition in Definition \ref{KK0cycle}. So we must check that, for $a\in C_0(\underline{E})\otimes C^*_r(G)$, the operator $\tilde{F}\tilde{\pi}_+(a)-\tilde{\pi}_-(a)\tilde{F}$ is compact. By linearity and density, we may assume $a=f\otimes\delta_h$, for some $f\in C_0(\underline{E}),h\in G$. For $\xi\in\ell^2(G,\HH^+), g\in G$, we then have:
$$([\tilde{F}\tilde{\pi}_+(f\otimes\delta_h)-\tilde{\pi}_-(f\otimes\delta_h)\tilde{F}]\xi)(g)=(g\cdot F)(\tilde{\pi}_+((f\otimes\delta_h)\xi)(g))-\pi_-(f)(\tilde{F}\xi(gh))$$
$$=(g\cdot F)\pi_+(f)(\xi(gh))-\pi_-(f)(gh\cdot F)(\xi(gh))$$
$$=[(g\cdot F)\pi_+(f)-\pi_-(f)(g\cdot F)](\xi(gh)) + [\pi_-(f)((g\cdot F)-(gh\cdot F)](\xi(gh))$$
So, up to the operator of right translation by $h$, the operator $\tilde{F}\tilde{\pi}_+(f\otimes\delta_h)-\tilde{\pi}_-(f\otimes\delta_h)\tilde{F}$ is given by the block-diagonal operator
$$\bigoplus_{g\in G}([(g\cdot F)\pi_+(f)-\pi_-(f)(g\cdot F)] + [\pi_-(f)((g\cdot (F-(h\cdot F)))])$$
We must show that each summand is compact, with norm going to 0 for $g\ra\infty$. By property $(\gamma)$, the term in the first bracket is indeed a compact operator with norm going to 0 for $g\ra\infty$. For the term in the 2nd bracket: the operator $F-(h\cdot F)$ is compact, as $x$ is a cycle for $KK_0^G(\C,\C)$, so lemma \ref{nondeg} applies to conclude that the norm of $\pi_-(f)((g\cdot (F-(h\cdot F)))$ goes to 0 for $g\ra\infty$.

The conditions $(\tilde{F}^*\tilde{F}-1)\tilde{\pi}_+(a)$ and $(\tilde{F}\tilde{F}^*-1)\tilde{\pi}_-(a)$ compact, for $a\in C_0(\underline{E})\otimes C^*_r(G)$, are checked in a similar way.

\subsection{Some results by Nishikawa \cite{Nish}}

The next result clarifies the link between property $(\gamma)$ and the existence of a $\gamma$-element. Compare with the first part of Theorem \ref{Tu}.

\begin{Prop}\begin{enumerate}
\item If $x$ is a cycle for $KK_0^G(\C,\C)$ with property $(\gamma)$, then $x$ is unique up to homotopy. Moreover $x$ defines an idempotent in the ring $KK_0^G(\C,\C)$ (as a consequence of uniqueness).
\item If $G$ admits a $\gamma$-element, then some cycle homotopic to $\gamma$ satisfies property $(\gamma)$.
\end{enumerate}\end{Prop}

The next result shows that an element with property $(\gamma)$ is nearly as good as a $\gamma$-element. Compare with the second part of Theorem \ref{Tu}.

\begin{Thm}\label{Shin}  Assume that there exists a cycle $x$ with property $(\gamma)$.
\begin{enumerate}
\item For any $G-C^*$-algebra $B$, we have $\nu_B^x\circ\mu_B=1_{KK_i^G(\underline{E},B)}$, so that $\mu_B$ is split injective. In particular the strong Novikov conjecture (Conjecture \ref{SNC}) holds.
\item If $x$ is homotopic to $1_{\C}$ in $KK_0^G(\C,\C)$, then the Baum-Connes conjecture with coefficients (Conjecture \ref{BCcoeff}) holds for $G$.
\end{enumerate}\end{Thm}

{\bf Ideas for  the proof:} There are 4 steps.

\begin{enumerate}
\item One first checks that $\mu_B\circ\nu_B^x$ coincides with the module action of $x$ on $K_i(B\rtimes_r G)$. More precisely, for every $y\in K_i(B\rtimes_r G)$, we have $(\mu_B\circ\nu_B^x)(y)=y\otimes_{B\rtimes_r G} j_G(\tau_B(x))$. In particular $\mu_B\circ\nu_B^x=1_{K_i(B\rtimes_r G)}$ if $x=1_{\C}$.

\item Assume that $\mu_B$ is an isomorphism. Then we have an obviously commutative diagram, with vertical arrows being isomorphisms:
$$\left.\begin{array}{ccc}KK_i^G(C_0(\underline{E}),B) & \stackrel{\nu_B^x\circ\mu_B}{\longrightarrow} & KK_i^G(C_0(\underline{E}),B)\\ & & \\ 
\mu_B\downarrow\simeq&  & \simeq\downarrow\mu_B \\ & & \\ 
K_i(B\rtimes_rG) & \stackrel{\mu_B\circ\nu_B^x}{\longrightarrow} & K_i(B\rtimes_r G)
\end{array}\right.$$
By the first step, the bottom horizontal map is the module action of $x$. But the diagram also commutes if the top horizontal map is replaced by the module action of $x$. As vertical maps are isomorphisms, we deduce that the top horizontal map coincides with the module action of $x$. 

\item We now remove the assumption in the second step, i.e we show that $\nu_B^x\circ\mu_B$ is the module action of $x$ on $KK_i(C_0(\underline{E}),B)$ for any $G-C^*$-algebra $B$. By a result of R. Meyer and R. Nest \cite{MeNe}, there exists a proper $G-C^*$-algebra $\tilde{B}$ and $y\in KK_0^G(\tilde{B},B)$ such that the top horizontal map of the commutative diagram

$$\left.\begin{array}{ccc}
KK_i^G(C_0(\underline{E}),\tilde{B}) & \stackrel{\otimes_{\tilde{B}} y}{\longrightarrow} & KK_i^G(C_0(\underline{E}),B)\\ & & \\ 
\mu_{\tilde{B}}\downarrow&  & \downarrow\mu_B \\ & & \\ 
K_i(\tilde{B}\rtimes_rG) & \stackrel{\otimes_{\tilde{B}\rtimes_r G}  j_G(y)}{\longrightarrow} & K_i(B\rtimes_r G)
\end{array}\right.$$

is an isomorphism (also $\mu_{\tilde{B}}$ is an isomorphism, by Theorem \ref{Tuproper}). Applying $\nu_{\tilde{B}}^x$ on the left and $\nu_B^x$ on the right, we get the commutative diagram

$$\left.\begin{array}{ccc}
KK_i^G(C_0(\underline{E}),\tilde{B}) & \stackrel{\otimes_{\tilde{B}} y}{\longrightarrow} & KK_i^G(C_0(\underline{E}),B) \\ & & \\ 
\nu_{\tilde{B}}^x\circ\mu_{\tilde{B}}\downarrow&  & \downarrow\nu_B^x\circ\mu_B \\ & & \\ 
KK_i^G(C_0(\underline{E}),\tilde{B}) & \stackrel{\otimes_{\tilde{B}} y}{\longrightarrow} & KK_i^G(C_0(\underline{E}),B)
\end{array}\right.$$

with horizontal maps being isomorphisms. By the second step, the left vertical map is the module action of $x$ on $KK_i^G(C_0(\underline{E}),\tilde{B})$. Since the diagram also commutes when the right vertical map is replaced by the module action of $x$, we deduce that $\nu_B^x\circ\mu_B$ is indeed the module action of $x$ on $KK_i^G(C_0(\underline{E}),B) $.

\item Because $\underline{E}$ is $G$-proper, a cycle $z$ for $KK_0^G(\C,\C)$ acts as the identity on $KK_j^G(C_0(\underline{E}),B)$ for every $G-C^*$-algebra $B$, if and only if it satisfies $Rest_G^F z=\epsilon_F$ for every compact subgroup $F$ of $G$.\epr
\end{enumerate}

\begin{Ex} Let $X$ be either Euclidean space $\mathbb{E}^n$ or a Riemannian symmetric space of the non-compact type. In this case, we know that there is a $\gamma$-element represented in the unbounded operator picture by the operator $D$ in formula (\ref{DdualD}), namely
$$D=d+d^*+ext(d\rho)+int(d\rho),$$
with $\rho(x)=\frac{d_M(x_0,x)^2}{2}$. Pass to the bounded picture by $F:=D(1+D^*D)^{-1/2}$. Then Nishikawa shows directly that the cycle $(L^2(\Lambda^+T^*M),L^2(\Lambda^-T^*M),F)$ has property $(\gamma)$. For this he appeals to the integral formula
$$F=\frac{2}{\pi}\int_0^\infty D(1+D^*D+\lambda^2)^{-1}\,d\lambda,$$
based on the identity $\frac{s}{\sqrt{1+s^2}}=\frac{2}{\pi}\int_0^\infty\frac{s}{1+s^2+\lambda^2}\,d\lambda$ for $s>0$.
\end{Ex}

Nishikawa is then able to reprove that the following groups satisfy the Baum-Connes conjecture with coefficients:
\begin{itemize}
\item Closed subgroups of $Isom(\mathbb{E}^n)$.
\item Groups acting properly co-compactly on trees.
\item (in joint work with J. Brodzki, E. Guentner and N. Higson, see \cite{BGHN}) Groups acting properly co-compactly on finite-dimensional CAT(0) cubical complexes.
\end{itemize}

\section{The assembly map for some discrete groups}

In this final section we gather some examples of {\it discrete} groups where the assembly map $\mu_G$ can be explicitly described.

If a torsion-free group $G$ has a model for $BG$ with dimension $\leq 3$, then $K_0(BG)\simeq H_0(BG,\Z)\oplus H_2(BG,\Z)$ and $K_1(BG)\simeq H_1(BG,\Z)\oplus H_3(BG,\Z)$ (see \cite{Matthey}; note that homology of the classifying space identifies with standard group homology). Of course $H_0(BG,\Z)=\Z$, generated by $[*]$, the inclusion of a base-point.

\subsection{The group $\Z$}

For $G$ infinite cyclic, say with generator $t$, we have on the left-hand side:
\begin{itemize}
\item $H_0(S^1,\Z)=\Z$ generated by $[*]$.
\item $H_1(S^1,\Z)=\Z$, generated by the fundamental cycle $[c]$.
\end{itemize}

On the right-hand side we have $C^*(\Z)=C(S^1)$ (by Example \ref{abelian1}), so 
\begin{itemize}
\item $K_0(C^*(\Z))=\Z$ generated by $[1]$.
\item $K_1(C^*(\Z))=\Z$ generated by $[t]$, the class of the generator $t$.
\end{itemize}
Then $\mu_\Z[*]=[1]$ and $\mu_\Z[c]=[t]$.

\subsection{The free group on $n$ generators}

For $n\geq 2$, let $F_n$ be the free group on the $n$ generators $t_1,...,t_n$. Then $BG$ is a wedge (=bouquet) of $n$ circles $c_1,...,c_n$, so that on the left-hand side:
\begin{itemize}
\item $H_0(BG,\Z)=\Z$ generated by $[*]$.
\item $H_1(BG,\Z)=\Z^n$ generated by $[c_1],...,[c_n]$.
\end{itemize}

On the left-hand side we have by a deep result of Pimsner-Voiculescu \cite{PV}:
\begin{itemize}
\item $K_0(C^*_r(F_n))=\Z$, generated by $[1]$.
\item $K_1(C^*_r(F_n))=\Z^n$, generated by $[t_1],...,[t_n]$, the classes of the free generators $t_1,...,t_n$.
\end{itemize}

Then $\mu_{F_n}[*]=[1]$ and $\mu_{F_n}[c_i]=[t_i]$ for $i=1,...,n$.

Note that a consequence of the Pimsner-Voiculescu computation is that $C^*_r(F_m)$ is not isomorphic to $C^*_r(F_n)$ for $m\neq n$: that was a longstanding open question in 1982!

\subsection{Finite groups}

Let $F$ be a finite group. By the first part of Example \ref{universalex}, we may take $\underline{E}F$ to be a point, and $K_0^F(\underline{E}F)=R(F)$, the additive group of the representation ring of $F$; and $K_1^F(\underline{E}F)=0$. View $R(F)$ as the free abelian group on $\hat{F}$, the dual of $F$.

On the other hand 
$$C^*_r(F)=\C F=\oplus_{\pi\in\hat{F}}M_{\dim\pi}(\C).$$
So that $K_1(C^*_r(F))=0$ and 
$$K_0(C^*_r(F))=\oplus_{\pi\in\hat{F}}\Z\cdot[e_\pi],$$
where $e_\pi$ is a minimal projection in $M_{\dim\pi}(\C)$. If $\dim\pi\geq 2$, there is no canonical choice for $e_\pi$, although the class $[e_\pi]$ is canonical. At the other extreme for $\pi=\epsilon_F$, the corresponding projection is 
\begin{equation}\label{p_F}
p_F:=\frac{1}{|F|}\sum_{g\in F}g.
\end{equation}

Then $\mu_F:R(F)\ra K_0(C^*_r(F))$ is given by $\mu_F[\pi]=[e_\pi]$ for $\pi\in\hat{F}$.

\subsection{Some locally finite groups}

A locally finite group $B$ is the union $B=\cup_{n>0} B_n$ of an increasing family of finite groups $B_n$, with $B_n\subset B_{n+1}$. By the last part of Example \ref{universalex}, we know that $\underline{E}B$ is a tree that fibers onto $\N$.

Now $$C^*(B)=\lim_{\longrightarrow} C^*(B_n)$$ is an AF-algebra\footnote{AF stands for {\it approximately finite}, i.e. an inductive limit of finite-dimensional $C^*$-algebras; those were fashionable in the 1980's.}. Since K-theory commutes with inductive limits, we have
$$K_i(C^*(B))=\lim_{\longrightarrow} K_i(C^*(B_n))\;(i=0,1);$$
in particular $K_i(C^*(B))=0$. By functoriality of the assembly map, we get that $\mu_B: K_i^G(\underline{E}B)\ra K_i(C^*(B))\,(i=0,1)$ is an isomorphism.

It may be interesting to have an explicit description of $\mu_B$, at least in special cases. So let $(F_n)_{n>0}$ be a collection of finite groups, set $B_n:=\oplus_{k=1}^nF_k$ and 
$$B:=\lim_{\longrightarrow}B_n=\bigoplus_{k=1}^\infty F_k.$$ 
We embed $\widehat{B_n}$ into $\widehat{B_{n+1}}$ via
$$s_n:\widehat{B_n}\ra\widehat{B_{n+1}}:\pi\mapsto\pi\otimes\epsilon_{F_{n+1}}.$$

\begin{Prop}\label{FPV1}(see \cite{FPV17}) The group $K_0^B(\underline{E}B)$ is naturally isomorphic to the free abelian group on the inductive limit of the system $(\widehat{B_n},s_n)_{n>0}$.
\end{Prop}

To be even more specific, assume that all $F_n$'s are equal to the same finite group $F$ with $|F|>1$, so that $B_n=F^n$ and $B=F^{(\N)}$. Consider the {\it non-unital} *-homomorphism
$$j_n:C^*(B_n)\ra C^*(B_{n+1})=C^*(B_n)\otimes C^*(F):x\mapsto x\otimes p_F$$
with $p_F$ as in (\ref{p_F}). It turns out that 
$$(j_n)_*\circ\mu_{B_n}=\mu_{B_{n+1}}\circ s_n.$$
With the help of Proposition \ref{FPV1}, this will be useful in describing $\mu_B$ explicitly.

For $X$ a pointed set, with base-point $o$, denote by $X^{(\N)}$ the set of maps $\N\ra X$ which take the value $o$ for all but finitely many integers. The set $\hat{F}$ is pointed by $\epsilon_F$; the set $Min(F)$ of classes of minimal projections in $C^*(F)$ is pointed by $p_F$.

\begin{Prop}\label{FPV2}(see \cite{FPV17}). Let $F$ be a non-trivial finite group, $B=F^{(\N)}$.
\begin{enumerate} 
\item The free abelian group on $\hat{F}^{(\N)}$ is naturally isomorphic to $K_0^G(\underline{E}B)$ with isomorphism given by
$$\pi\in\hat{F}^{(\N)}\mapsto\otimes_{k\in Supp(\pi)}\pi_k\in R(\oplus_{k\in Supp(\pi)}F_k)\subset K_0^B(\underline{E}B).$$
\item The free abelian group on $Min(F)^{(\N)}$ is naturally isomorphic to $K_0(C^*(B))$, with isomorphism given by
$$e\in Min(F)^{(\N)}\mapsto \otimes_{k\in Supp(e)} e_k\in Min(\oplus_{k\in Supp(e)} F_k)\subset K_0(C^*(B)).$$
\item For $\pi\in\hat{F}^{(\N)}$ we have $\mu_B(\pi)=\otimes_{k\in Supp(\pi)}\mu_F(\pi_k)\in K_0(C^*(B))$.
\end{enumerate}\end{Prop}

\subsection{Lamplighter groups}

For $F$ a non-trivial finite group, consider the {\it lamplighter group} $L_F:=F\wr\Z=(\oplus_\Z F)\rtimes \Z$ where $\Z$ acts on itself by translations. The normal subgroup $B:=\oplus_\Z F$ is a locally finite group precisely of the form studied in Proposition \ref{FPV2}. Let $t$ denote one generator of $\Z$, acting as the bilateral shift on $B$.

\begin{Prop}\label{lamplighter} (see \cite{FPV17})\begin{enumerate}
\item $K_0(C^*(L_F))$ is free abelian of countable rank, with an explicit basis consisting in projections from $C^*(B)$.
\item $K_1(C^*(L_F))=\Z$ generated by the class $[t]$ of $t$.
\item There exists a 2-dimensional model for for $\underline{E}L_F$ and $\mu_{L_F}$ is an isomorphism that can be explicitly described.
\end{enumerate}\end{Prop}

Of course, as $L_F$ is amenable, the fact that $\mu_{L_F}$ is an isomorphism follows from the last item in Theorem \ref{gammaexists}. The point of Proposition \ref{lamplighter} is to make this isomorphism explicit. 

Let us close by describing the 2-dimensional model for $\underline{E}L_F$, in a slightly more explicit way than in \cite{FPV17}. For $k\in\Z,n\in\N$, set $F_k:=F$ and $B_n :=\oplus_{k=-n}^n F_k$, so that $B=\cup_{n=0}^\infty B_n's$. Let $\underline{E}B$ be as in the last part of Example \ref{universalex}, i.e a tree with vertex set $V:=\coprod_{n\geq 0} B/B_n$ and edge set $\{(gB_n,gB_{n+1}): g\in B, n\geq 0\}$. Observe that, for $n>0$, every vertex in $B/B_n$ has $|F|^2=[B_n:B_{n-1}]$ incoming edges, and just one outgoing edge.

The vertex set of $\underline{E}L_F$ is $V\times\Z$. Set $V_k:=V\times\{k\}$, the {\it $k$-th level} of $\underline{E}L_F$. The group $\Z$ acts by translation on the second factor. The element $b\in B$ acts on $V_k$ as $t^k(b)$; using the defining relation $t^{-1}bt=t(b)$, one checks that the $B$- and $\Z$- actions combine to extend to an action of $L$. There are horizontal edges (those coming from $\underline{E}B$ in each level), and oblique edges defined by connecting $(gB_{n-1},k)$ in the $k$-th level to $(t(g)B_{n},k+1)$ in the $(k+1)$-th level ($g\in B,n\in\N,k\in\Z$; this is well-defined because $t(B_{n-1})\subset B_n$). The set of oblique edges is invariant under the $L$-action. Finally, we glue a 2-cell on each quadrilateral 
$$(gB_{n-1},k),(t(g)B_n,k+1), (t(g)B_{n+1},k+1), (gB_n,k)$$
consisting of 2 oblique and 2 horizontal edges ($g\in B,n\in\N,k\in\Z$). This way we get a 2-dimensional $\underline{E}L_F$.

\begin{Exo} Check the details in the above construction.
\end{Exo}
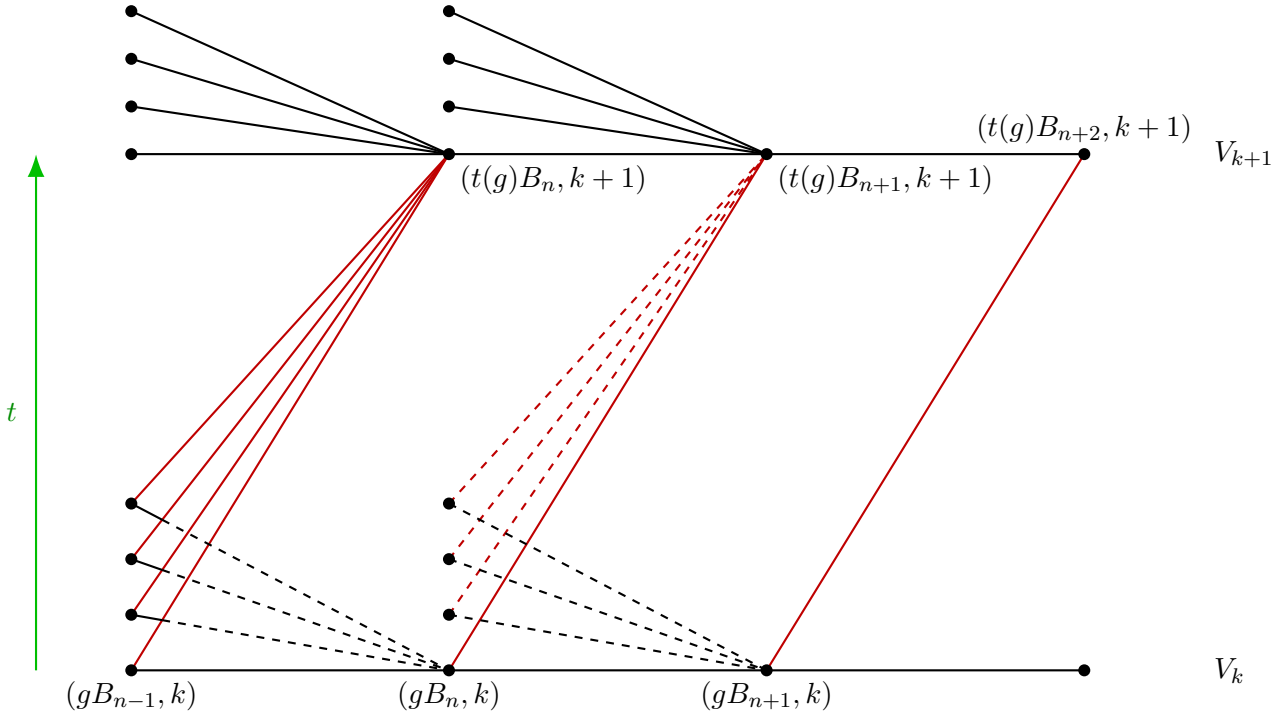
\begin{figure}[h]
\centering

\begin{tikzpicture}[
    scale=1.05,
    point/.style={circle, fill=black, inner sep=1.6pt},
    blackline/.style={black, thick},
    blackdash/.style={black, thick, dashed},
    redline/.style={red!75!black, thick},
    reddash/.style={red!75!black, thick, dashed},
    greenarrow/.style={green!75!black, thick, -{Latex[length=3mm]}},
    every node/.style={font=\small}
]

\hspace{-2cm}
\coordinate (A0) at (0,0);
    \coordinate (A01) at (0,0.7);
    \coordinate (A02) at (0,1.4);
    \coordinate (A03) at (0,2.1);
\coordinate (A1) at (4,0);
    \coordinate (A11) at (4,0.7);
    \coordinate (A12) at (4,1.4);
    \coordinate (A13) at (4,2.1);
\coordinate (A2) at (8,0);
\coordinate (A3) at (12,0);

\coordinate (B0) at (0,6.5);
    \coordinate (B01) at (0,7.1);
    \coordinate (B02) at (0,7.7);
    \coordinate (B03) at (0,8.3);
\coordinate (B1) at (4,6.5);
    \coordinate (B11) at (4,7.1);
    \coordinate (B12) at (4,7.7);
    \coordinate (B13) at (4,8.3);
\coordinate (B2) at (8,6.5);
\coordinate (B3) at (12,6.5);

\draw[blackline] (A0) -- (A3);
\draw[blackline] (B0) -- (B3);

\node[below] at (A0) {$(g B_{n-1},k)$};
\node[below] at (A1) {$(g B_n,k)$};
\node[below] at (A2) {$(g B_{n+1},k)$};
\node[right] at (A3) {};
\node[right] at (13.5,0) {$V_k$};

\node[below right] at (B0) {};
\node[below right] at (B1) {$(t(g) B_{n},k+1)$};
\node[below right] at (B2) {$(t(g) B_{n+1},k+1)$};
\node[above] at (B3) {$(t(g)B_{n+2}, k+1)$};
\node[right] at (13.5,6.5) {$V_{k+1}$};

\foreach \P in {A0, A01,A02,A03}
{
    \draw[redline] ($(\P)$) -- ($(B1)$);
}

\foreach \P in {A11,A12,A13}
{
    \draw[reddash] ($(\P)$) -- ($(B2)$);
}

\foreach \P in {A1}
{
    \draw[redline] ($(\P)$) -- ($(B2)$);
}

\foreach \P in {A2}
{
    \draw[redline] ($(\P)$) -- ($(B3)$);
}


\coordinate (C1) at (0.3889,0.6319);
\coordinate (C2) at (0.3889,1.2639);
\coordinate (C3) at (0.3889,1.8958);

\draw[blackline] (A01) -- (C1);
\draw[blackdash] (C1) -- (A1);

\draw[blackline] (A02) -- (C2);
\draw[blackdash] (C2) -- (A1);

\draw[blackline] (A03) -- (C3);
\draw[blackdash] (C3) -- (A1);

\foreach \P in {A11,A12,A13}
{
    \draw[blackdash] ($(\P)$) -- ($(A2)$);
}

\foreach \P in {B01,B02,B03}
{
    \draw[blackline] ($(\P)$) -- ($(B1)$);
}

\foreach \P in {B11,B12,B13}
{
    \draw[blackline] ($(\P)$) -- ($(B2)$);
}

\foreach \P in {A0,A1,A2,A3,B0,B1,B2, B3, A01, A02, A03, A11,A12,A13,B01,B02,B03,B11,B12,B13}
    \node[point] at (\P) {};

\draw[greenarrow] (-1.2,0.0) -- (-1.2,6.5);
\node[green!55!black, left] at (-1.3,3.25) {$t$};

\end{tikzpicture}

\caption{$\underline{E} L_F$ for $F=C_2$}
\end{figure}

\noindent
Author's address:\\

\noindent
Institut de math\'ematiques\\
Universit\'e de Neuch\^atel\\
11 Rue Emile Argand - Unimail\\
CH-2000 Neuch\^atel - Switzerland\\
alain.valette@unine.ch

\end{document}